%% file: main.tex
\documentclass[11pt,twoside,reqno]{amsart}

\usepackage[utf8]{inputenc}
\usepackage[T1]{fontenc}
\usepackage{import}

\input{preamble}

\title[Universality of zeros for $p$-adic Random Matrices and Polynomials]{The Resultant Distribution Method: Universality for $p$-adic Random Matrices and Polynomials}
\author{Jiahe Shen}
\date{\today}

\begin{document}

\thanks{The author thanks Roger Van Peski for carefully reading the draft and providing helpful comments. The author acknowledges support from Ivan Corwin's NSF grant DMS-2246576 and Simons Investigator grant 929852.}

\maketitle

\begin{abstract}
We prove universality of limiting local eigenvalue statistics for random
matrices over $\Z_p$. In previous work of the author and Van Peski
\cite{shen2026eigenvalues}, the limiting eigenvalue correlation functions
of additive Haar random matrices were studied in arbitrary finite extensions
of $\Q_p$. The same Haar random matrix model plays a central role in the
Ellenberg-Jain-Venkatesh heuristic for zeros of $p$-adic $L$-functions.
We show that its limiting local eigenvalue statistics are unchanged for a
broad class of random matrices with independent entries satisfying a mild
non-concentration condition. Thus the random matrix predictions underlying
the Ellenberg-Jain-Venkatesh heuristic are not artifacts of the particular
Haar ensemble, but instead reflect universal limiting eigenvalue statistics.
In this sense, our results provide additional theoretical support for the
robustness of their random matrix heuristic.

Our proof is based on a new framework, which we call the
\emph{resultant distribution method}. The method recovers limiting laws and
root statistics of $p$-adic polynomials from the distributions of their
resultant valuations against fixed test polynomials, together with suitable
degree estimates. As a second application, we consider random $p$-adic
polynomials with independent coefficients satisfying a mild non-concentration
condition. Caruso \cite{caruso2022zeroes} determined the joint root
correlation functions of the Haar coefficient model over finite extensions
of $\Q_p$. We prove that, for roots of absolute value one, these limiting
correlation functions are universal and persist for a broad class of
independent coefficient distributions.
\end{abstract}



\tableofcontents

\input{Introduction}
\input{Preliminaries}
\input{The_resultant_distribution_method}

\input{The_random_matrix_case}

\input{Growing_Surjection_Moments_and_Moment_Tightness}
\input{The_random_polynomial_case}

\bibliographystyle{plain}
\bibliography{references.bib}

\end{document}

%% file: preamble.tex
\usepackage{anysize}
\usepackage{amsmath,amsthm,amscd,amssymb}
\usepackage{mathtools}
\usepackage{stmaryrd}
\usepackage{wasysym}
\usepackage{mathrsfs}
\usepackage{mathdots}
\usepackage{dsfont}
\usepackage{bm}
\usepackage{scalerel,stackengine}

\usepackage{xcolor}
\usepackage{graphicx}
\usepackage{enumitem}
\usepackage{soul}
\usepackage{float}
\usepackage{makecell}
\usepackage{setspace}
\usepackage{comment}
\usepackage{csquotes}
\usepackage[all]{xy}
\usepackage{tikz}
\usepackage{tikz-cd}

\usepackage[myheadings]{fullpage}

\usepackage[pagebackref=true,hypertexnames=false]{hyperref}
\hypersetup{colorlinks=true,linkcolor=black,citecolor=black}
\renewcommand*{\backref}[1]{}
\renewcommand*{\backrefalt}[4]{({\tiny%
   \ifcase #1 Not cited.%
         \or Cited on page~#2.%
         \else Cited on pages #2.%
   \fi%
   })}
\usepackage{cleveref}

\DeclareFontFamily{U}{mathb}{}

\DeclareFontShape{U}{mathb}{m}{n}{
   <-5.5>  mathb5
   <5.5-6.5> mathb6
   <6.5-7.5> mathb7
   <7.5-8.5> mathb8
   <8.5-9.5> mathb9
   <9.5-11>  mathb10
   <11->     mathb12
}{}

\DeclareSymbolFont{mathb}{U}{mathb}{m}{n}

\DeclareMathSymbol{\lefttorightarrow}{3}{mathb}{"FC}
\DeclareMathSymbol{\righttoleftarrow}{3}{mathb}{"FD}

\makeatletter
\newcommand{\acts}{%
  \mathrel{\mathpalette\@acts\righttoleftarrow}}
\newcommand{\actedby}{%
  \mathrel{\mathpalette\@acts\lefttorightarrow}}
\newcommand{\@acts}[2]{\reflectbox{$\m@th#1#2$}}
\makeatother

\numberwithin{equation}{section}
\newcommand\mtop{.95in}
\newcommand\mbottom{.95in}
\newcommand\mleft{1in}
\newcommand\mright{1in}
\usepackage[top = \mtop, bottom = \mbottom, left = \mleft, right=\mright]{geometry}
\DeclareMathOperator{\val}{val}
\DeclareMathOperator{\Mat}{Mat}

\newtheorem{thm}{Theorem}[section]

\newtheorem{prop}[thm]{Proposition}
\newtheorem{lemma}[thm]{Lemma}

\newtheorem{cor}[thm]{Corollary}
\theoremstyle{definition}
\newtheorem{defi}[thm]{Definition}

\newtheorem{rmk}[thm]{Remark}
\newtheorem{remark}[thm]{Remark}

\crefname{thm}{theorem}{theorems}
\Crefname{thm}{Theorem}{Theorems}
\crefname{theorem}{theorem}{theorems}
\Crefname{theorem}{Theorem}{Theorems}
\crefname{prop}{proposition}{propositions}
\Crefname{prop}{Proposition}{Propositions}
\crefname{lemma}{lemma}{lemmas}
\Crefname{lemma}{Lemma}{Lemmas}
\crefname{conj}{conjecture}{conjectures}
\Crefname{conj}{Conjecture}{Conjectures}
\crefname{cor}{corollary}{corollaries}
\Crefname{cor}{Corollary}{Corollaries}
\crefname{defi}{definition}{definitions}
\Crefname{defi}{Definition}{Definitions}
\crefname{notation}{notation}{notations}
\Crefname{notation}{Notation}{Notations}
\crefname{rmk}{remark}{remarks}
\Crefname{rmk}{Remark}{Remarks}
\crefname{remark}{remark}{remarks}
\Crefname{remark}{Remark}{Remarks}

\newcommand\reallywidehat[1]{%
\savestack{\tmpbox}{\stretchto{%
  \scaleto{%
    \scalerel*[\widthof{\ensuremath{#1}}]{\kern-.6pt\bigwedge\kern-.6pt}%
    {\rule[-\textheight/2]{1ex}{\textheight}}
  }{\textheight}%
}{0.5ex}}%
\stackon[1pt]{#1}{\tmpbox}%
}
\DeclareSymbolFont{bbold}{U}{bbold}{m}{n}
\DeclareSymbolFontAlphabet{\mathbbold}{bbold}

\makeatletter
\def\@tocline#1#2#3#4#5#6#7{\relax
  \ifnum #1>\c@tocdepth 
  \else
    \par \addpenalty\@secpenalty\addvspace{#2}%
    \begingroup \hyphenpenalty\@M
    \@ifempty{#4}{%
      \@tempdima\csname r@tocindent\number#1\endcsname\relax
    }{%
      \@tempdima#4\relax
    }%
    \parindent\z@ \leftskip#3\relax \advance\leftskip\@tempdima\relax
    \rightskip\@pnumwidth plus4em \parfillskip-\@pnumwidth
    #5\leavevmode\hskip-\@tempdima
      \ifcase #1
       \or\or \hskip 1em \or \hskip 2em \else \hskip 3em \fi%
      #6\nobreak\relax
    \hfill\hbox to\@pnumwidth{\@tocpagenum{#7}}\par
    \nobreak
    \endgroup
  \fi}
\makeatother

\newcommand{\R}{\mathbb{R}}
\newcommand{\Z}{\mathbb{Z}}
\newcommand{\Q}{\mathbb{Q}}

\newcommand{\F}{\mathbb{F}}

\newcommand{\E}{\mathbb{E}}

\newcommand{\mc}{\mathcal}

\newcommand{\Y}{\mathbb{Y}}

\DeclareMathOperator{\Hom}{Hom}
\DeclareMathOperator{\Sur}{Sur}

\DeclareMathOperator{\Res}{Res}
\DeclareMathOperator{\new}{new}
\DeclareMathOperator{\Disc}{Disc}

\DeclareMathOperator{\Haar}{Haar}

\DeclareMathOperator{\Cok}{Cok}

\DeclareMathOperator{\poly}{poly}

\DeclareMathOperator{\Gal}{Gal}
\DeclareMathOperator{\Aut}{Aut}

\DeclareMathOperator{\ord}{ord}

\DeclareMathOperator{\Span}{span}

\DeclareMathOperator{\GL}{GL}

\DeclareMathOperator{\GSp}{GSp}

\DeclareMathOperator{\diag}{diag}

\DeclareMathOperator{\rank}{rank}



%% file: Introduction.tex
\section{Introduction}
\label{sec: introduction}

This paper is motivated by two recent developments in the study of random
roots over $p$-adic fields. In the random matrix setting, the author and
Van Peski \cite{shen2026eigenvalues}  developed a theory of eigenvalue correlations for additive Haar random matrices over $\Z_p$, including their limiting statistics in arbitrary finite extensions of $\Q_p$. In the random polynomial setting, Caruso
\cite{caruso2022zeroes} obtained general correlation formulas for polynomials with independent Haar-distributed coefficients. These results provide natural uniform reference models in two rather different settings. The purpose of the present paper is to understand to what extent their
limiting root statistics persist beyond the Haar case. Both problems are treated using the same basic approach, which we call the \emph{resultant distribution method}. This method reconstructs limiting root distributions and statistics from the distributions of suitable resultant valuations.

\subsection{Universality of eigenvalue correlations for $p$-adic random matrices}

On the classical random matrix theory side, the most integrable ensembles
possess remarkably strong exact structures. For the Gaussian unitary
ensemble, for instance, the eigenvalues form a determinantal point process:
all joint correlation functions are expressed as determinants of a single
correlation kernel. Such determinantal formulas go back to the classical
work of Dyson
\cite{dyson1962statisticalI,dyson1962statisticalIII} and are developed
systematically in Mehta's monograph \cite{mehta2004random}. More generally,
the theory of determinantal point processes was developed by Macchi
\cite{macchi1975coincidence} and surveyed extensively by Soshnikov
\cite{soshnikov2000determinantal}. These exact structures make it possible
to obtain remarkably precise information about eigenvalue statistics and
their asymptotic behavior, including local limits in the bulk and at the edge.

A striking feature of modern random matrix theory is that many of these
limiting statistics survive far beyond the exactly solvable ensembles.
For matrices with more general entry distributions, the determinantal
structure is typically lost, but the local statistics often remain
unchanged after the appropriate normalization. This universality phenomenon
has been established under increasingly general assumptions, notably in
work of Tao-Vu \cite{tao2011random} and of Erd\H{o}s-Yau-Yin
\cite{erdHos2012bulk}, among many others.

The theory of eigenvalues of random matrices over $p$-adic fields is much
more recent. A natural starting point is the additive Haar ensemble on
$\Mat_n(\Z_p)$, which provides the canonical uniform model in this setting.
Motivated in part by the random matrix heuristic of
Ellenberg-Jain-Venkatesh \cite{ellenberg2011modeling}, in previous work
with Van Peski \cite{shen2026eigenvalues} we studied matrices
$$
A_n^{\Haar}\in\Mat_n(\Z_p)
$$
with independent additive Haar-distributed entries. We obtained the joint
eigenvalue distributions at finite matrix size and studied their asymptotic
behavior as $n\to\infty$.

More precisely, for any finite extensions
$K_1,\ldots,K_m$ of $\Q_p$, we introduced limiting correlation functions
$$
\rho_{K_1,\ldots,K_m}^{(\infty)}
:
\mc{O}_{K_1}^{\new}
\times\cdots\times
\mc{O}_{K_m}^{\new}
\longrightarrow
\R_{\geq 0},
$$
which describe the limiting density of ordered configurations of
eigenvalues lying near prescribed points in the extensions
$K_1,\ldots,K_m$, with the points belonging to distinct Galois orbits.
Equivalently, integrating these functions over a region gives the limiting
expected number of such eigenvalue configurations in that region. We refer
the reader to \Cref{subsec: Haar random matrix correlations} for a more
detailed discussion of these limiting correlation functions and their
precise relation to eigenvalue-counting statistics.

This naturally raises the analogue of the classical universality question:
to what extent do these limiting statistics depend on the exact Haar
distribution of the entries? The main result of the present paper shows
that they are in fact robust under substantial changes of the entry
distribution. To formulate the class of distributions for which this holds,
we use the following non-concentration condition.

\begin{defi}
\label{defi: epsilon balanced}
We say that a random $p$-adic integer $\xi\in\Z_p$ is
$\epsilon$\emph{-balanced} if
$$
\mathbf{P}(\xi\equiv r\pmod p)\leq 1-\epsilon
$$
for every $r\in\{0,1,\ldots,p-1\}$.
\end{defi}

To formulate joint eigenvalue statistics, we allow the eigenvalues to lie
in possibly different finite extensions of $\Q_p$. Let
$K_1,\ldots,K_m$ be finite extensions of $\Q_p$. For each $i$, write
$$
\mc{O}_{K_i}
:=
\left\{
x\in K_i: \|x\|\leq 1
\right\}
$$
for the ring of integers of $K_i$, and set
$$
\mc{O}_{K_i}^{\new}
:=
\left\{
x\in\mc{O}_{K_i}:\Q_p[x]=K_i
\right\}.
$$
Thus $\mc{O}_{K_i}^{\new}$ consists of those algebraic integers which
generate the prescribed extension $K_i$ over $\Q_p$. Restricting to this
set allows us to record not only the location of an eigenvalue, but also
the extension of $\Q_p$ that it generates.

Given a set
$$
U
\subseteq
\mc{O}_{K_1}^{\new}
\times\cdots\times
\mc{O}_{K_m}^{\new}
$$
and a nonzero polynomial $P\in\Z_p[x]$, we define $Z_U(P)$ to be the
number of tuples
$$
(x_1,\ldots,x_m)\in U
$$
such that each $x_i$ is a root of $P$ and the elements $x_1,\ldots,x_m$ lie in pairwise distinct Galois orbits over $\Q_p$.
The latter condition ensures that two coordinates do not represent the same algebraic root up to Galois conjugacy. 
For a matrix $A_n\in\Mat_n(\Z_p)$, let
$$
P_{A_n}(x):=\det(xI_n-A_n)
$$
denote its characteristic polynomial. Thus $Z_U(P_{A_n})$ counts
ordered configurations of eigenvalues of $A_n$ lying in the prescribed
extensions and region $U$.

With this notation, our main result states that the same limiting
correlation functions arise for every independent $\epsilon$-balanced
matrix ensemble.

\begin{thm}[Universality of eigenvalue correlations]
\label{thm: random matrix eigenvalue universality}
Let $K_1,\ldots,K_m$ be finite extensions of $\Q_p$, and let
$$
U\subseteq\mc{O}_{K_1}^{\new}
\times\cdots\times
\mc{O}_{K_m}^{\new}
$$
be clopen. For every $n\geq 1$, let
$A_n\in\Mat_n(\Z_p)$ be a random matrix whose entries are independent
and $\epsilon$-balanced. Then
\begin{equation}
\label{eq: main random matrix universality}
\lim_{n\to\infty}
\E\left[Z_U(P_{A_n})
\right]=\int_U\rho_{K_1,\ldots,K_m}^{(\infty)}(x_1,\ldots,x_m)\,dx_1\cdots dx_m.
\end{equation}
Here, the integral is taken with respect to the product Haar probability measure on
$\mc{O}_{K_1}\times\cdots\times\mc{O}_{K_m}$, and
$$\rho_{K_1,\ldots,K_m}^{(\infty)}:\mc{O}_{K_1}^{\new}
\times\cdots\times
\mc{O}_{K_m}^{\new}\longrightarrow \R_{\ge 0}$$
is the limiting eigenvalue correlation function of additive Haar random
matrices obtained in \cite[Section 6]{shen2026eigenvalues}.
\end{thm}

Although \Cref{thm: random matrix eigenvalue universality} is stated in terms of
expectations, our argument in fact yields the stronger conclusion that, for
every clopen $U$ as above, the random variable $Z_U(P_{A_n})$ itself
converges weakly to the corresponding limiting root-counting random variable
for the additive Haar ensemble; see the third part of
\Cref{thm: resultant distribution method introduction}.

We now record two concrete consequences of the universality theorem by
combining it with the explicit limiting expected counts of eigenvalues and
eigenvalue tuples obtained in \cite{shen2026eigenvalues}.

\begin{cor}
\label{thm: Zp first second moments universality}
For every $n\geq 1$, let
$A_n\in\Mat_n(\Z_p)$ be a random matrix whose entries are independent
and $\epsilon$-balanced. Then, we have
$$
\lim_{n\to\infty}\E[Z_{\Z_p}(P_{A_n})]=1
$$
and
$$
\lim_{n\to\infty}\E[Z_{\Z_p}(P_{A_n})^2]
=1+\sum_{k\geq 0}\frac{(-1)^k(1-p^{-1})(1+p^{-k-1})p^{-\frac{k^2+k}{2}}}{1-p^{-k^2-2k-2}}.
$$
\end{cor}

Our second example concerns eigenvalues generating a quadratic extension.

\begin{cor}
\label{thm: quadratic extension universality}
Let $K/\Q_p$ be a quadratic extension. For every $n\geq 1$, let
$A_n\in\Mat_n(\Z_p)$ be a random matrix whose entries are independent
and $\epsilon$-balanced. Then, we have
$$
\lim_{n\to\infty}\E[Z_{\mc{O}_K^{\new}}(P_{A_n})]=\begin{cases}
\displaystyle
\sum_{k\geq 0}\frac{(1-p^{-1})(1-p^{-k-1})(1+p^{-k^2-2k-3})
p^{-\frac{k^2+k}{2}}}{(1-p^{-k^2-2k-2})(1-p^{-k^2-2k-3})},
&K/\Q_p\text{ unramified},
\\
\left\lVert\Disc_{K/\Q_p}\right\rVert
\sum_{k\geq 0}\frac{(1-p^{-1})(1-p^{-2k-2})p^{-k^2-k}}{(1-p^{-k^2-2k-2})(1-p^{-k^2-2k-3})},& K/\Q_p\text{ ramified}.
\end{cases}
$$
\end{cor}

Universality phenomena for $p$-adic random matrices have previously
appeared primarily for singular numbers and cokernels. Van Peski
\cite{van2021limits} initiated a comparison between $p$-adic random matrix
theory and its real and complex counterparts, with $p$-adic singular numbers playing the role of logarithmic singular values. In a different direction,
Wood, Nguyen-Wood, and others established broad universality results for
random cokernels; see, for example,
\cite{wood2019random,nguyen2022random}. Particularly relevant to the present
work is Wood's theorem \cite{wood2019random}: for precisely the same class of random matrix ensembles considered here--matrices over $\Z_p$ with
independent $\epsilon$-balanced entries, the limiting distribution of the
cokernel is independent of the entry distributions and agrees with that of
the additive Haar ensemble.

The present paper reveals a different universality phenomenon within this
same underlying class of random matrices. Rather than the cokernel, we study
the eigenvalues themselves and show that their limiting local correlation
statistics are likewise insensitive to the entry distributions. Thus the
universality in \Cref{thm: random matrix eigenvalue universality} concerns a
new statistical dimension of the same random matrix ensemble. As far as we
are aware, it is the first universality theorem for local eigenvalue
statistics of $p$-adic random matrices. Unlike singular numbers or cokernels,
these statistics also retain the arithmetic information of the finite
extensions of $\Q_p$ generated by the eigenvalues.

Our result also strengthens the random matrix interpretation of the
Ellenberg-Jain-Venkatesh heuristic
\cite{ellenberg2011modeling}. Their heuristic is formulated in terms of a
Haar random matrix model \footnote{More precisely, Ellenberg-Jain-Venkatesh consider Haar-random elements of
$\GL_n(\Z_p)$ and of the corresponding symplectic similitude ensembles
$\GSp_\alpha(2n,\Z_p)$, and show in \cite[Theorem~4.1]{ellenberg2011modeling}
that the two ensembles induce the same limiting statistics for the associated
distinguished polynomial. In \cite[Section~7]{shen2026eigenvalues}, we further show that the limiting local eigenvalue statistics of Haar-random elements of $\GL_n(\Z_p)$ agree with those of additive Haar random matrices in
$\Mat_n(\Z_p)$, apart from the obvious restriction imposed by invertibility.
In particular, after translation by $1$, the local statistics relevant to the
Ellenberg-Jain-Venkatesh model coincide with those of the additive Haar
ensemble.}, whose limiting eigenvalue statistics were studied
in \cite{shen2026eigenvalues}. \Cref{thm: random matrix eigenvalue universality}
shows that the relevant limiting local statistics are not peculiar to the
Haar ensemble: the same limits arise for every random matrix ensemble with
independent $\epsilon$-balanced entries. Thus, to the extent that the random
matrix model captures the behavior predicted by the Ellenberg-Jain-Venkatesh heuristic, these predictions reflect a robust
universality phenomenon rather than a feature tied to one particular choice
of entry distribution.

\subsection{Idea of the proof}

A recurring principle in universality problems is to avoid comparing the random objects of interest directly. Instead, one first proves universality for a sufficiently rich family of auxiliary statistics and then recovers the desired limiting information from them. Classical examples of this philosophy include the moment method for Wigner matrices \cite{wigner1993characteristic1,wigner1993characteristic2,wigner1958distribution} and the surjection moment method pioneered by Wood in the study of random cokernels \cite{wood2017distribution}.

In parallel with these approaches, the main new ingredient of the present paper is what we call the \emph{resultant distribution method}. The guiding idea is to study a random polynomial indirectly through its resultants against a sufficiently rich collection of fixed test polynomials. More precisely, for a fixed monic polynomial $Z\in\Z_p[x]$, we consider the random variable
$$
\val\bigl(\Res(P,Z)\bigr)\in\Z_{\ge 0}\cup\{\infty\},
$$
where $\Res(P,Z)$ denotes the resultant of $P$ and $Z$, which measures the proximity between the roots of the two polynomials. We refer to \Cref{subsec:zpbar} for its precise definition and the basic properties.

A single such statistic contains only limited information about $P$. Taken over all suitable test polynomials $Z$, however, these resultant valuations determine the distribution of $P$ and, under appropriate control of its degree, also determine the limiting behavior of its root statistics.

We briefly introduce the setting in which the method is formulated. For the rest of the paper, we denote by $F_1=x,F_2,F_3,\ldots$ an enumeration of all monic irreducible polynomials over $\F_p$. The ring $\overline{\Z}_p$ of algebraic integers in $\overline{\Q}_p$ admits a corresponding decomposition
$$
\overline{\Z}_p=\bigsqcup_{i\geq 1}\mc{U}_i,
$$
where $\mc{U}_i$ consists of the algebraic integers whose minimal polynomial over $\Q_p$ reduces modulo $p$ to a power of $F_i$. We refer to $\mc{U}_i$ as the lifted subspace associated with $F_i$. In particular,
$$
\mc{U}_1=\{x\in\overline{\Q}_p:||x||_p<1\}
$$
is the open unit disk. Fix a finite set $S\subset\Z_{>0}$, and let $\mathscr{P}^S$ denote the space of monic polynomials over $\Z_p$ all of whose roots lie in $\bigcup_{i\in S}\mc{U}_i$. The following theorem summarizes the three main principles underlying the resultant distribution method. Precise versions, together with their proofs, are given in \Cref{sec: resultant distribution method}.

\begin{thm}[The resultant distribution method]
\label{thm: resultant distribution method introduction}
Fix a finite set $S\subseteq\Z_{>0}$.

\begin{enumerate}
\item
\emph{(Determination by resultant distributions).}
The law of a random polynomial $P^S\in\mathscr{P}^S$ is determined by the
distributions
$$
\val\left(\Res(P^S,Z)\right),
\qquad
Z\in\mathscr{P}^S.
$$
In particular, two random polynomials in $\mathscr{P}^S$ having the same
resultant-valuation distribution for every fixed $Z\in\mathscr{P}^S$ have
the same law.

\item
\emph{(Weak convergence from resultant distributions).}
Let $P_1^S,P_2^S,\ldots$ be random polynomials in $\mathscr{P}^S$. Suppose
that, for every fixed $Z\in\mathscr{P}^S$,
$$
\val\left(\Res(P_n^S,Z)\right)
$$
converges in distribution, and suppose that the sequence of random degrees
$\deg P_n^S$ is tight, in the sense that
$$
\lim_{D\to\infty}\sup_{n\geq 1}
\mathbf{P}\left(\deg P_n^S>D\right)=0.
$$
Then there exists a random polynomial $P_\infty^S\in\mathscr{P}^S$ such
that
$$
P_n^S\overset{d}{\longrightarrow}P_\infty^S.
$$

\item
\emph{(Convergence of root statistics and their expectations).}
Assume the hypotheses of the previous item, and let $P_\infty^S$ be the
resulting limiting random polynomial. Let $K_1,\ldots,K_m$ be finite
extensions of $\Q_p$, and let
$$
U\subseteq
\mc{O}_{K_1}^{\new}
\times\cdots\times
\mc{O}_{K_m}^{\new}
$$
be clopen. If $P_\infty^S$ is almost surely squarefree, then
$$
Z_U(P_n^S)\overset{d}{\longrightarrow}Z_U(P_\infty^S).
$$
If, in addition, the degree-tightness assumption is strengthened to
$$
\lim_{D\to\infty}\sup_{n\geq 1}
\E\left[
(\deg P_n^S)^m
\mathbf{1}_{\{\deg P_n^S>D\}}
\right]
=
0,
$$
then
$$
\lim_{n\to\infty}
\E\left[Z_U(P_n^S)\right]
=
\E\left[Z_U(P_\infty^S)\right].
$$
\end{enumerate}
\end{thm}

The first part of \Cref{thm: resultant distribution method introduction} is a uniqueness statement, while the second and third parts are two robustness statements of increasing strength. This philosophy is closely related to the uniqueness and robustness principles developed by Sawin-Wood \cite[Theorems 1.7 and 1.8]{sawin2022moment} for moment problems of random objects in a diamond category. Classical counterparts
for real random variables are the determinate moment problem of Stieltjes and
the moment-convergence theorem of Fr\'echet-Shohat \cite{stieltjes1894recherches,frechet1931proof}. The precise meaning of weak convergence here, with respect to the topology on the space of polynomials used throughout the paper, is clarified at the
beginning of \Cref{sec: resultant distribution method}.

We now explain why the method is especially well adapted to random
matrices. Let
$$
P_{A_n}(x)=\det(xI_n-A_n).
$$
For every monic $Z\in\Z_p[x]$, one has the elementary but crucial
identity
\begin{equation}
\label{eq: intro resultant matrix identity}
\Res(P_{A_n},Z)=\det\left(Z(A_n)\right).
\end{equation}
Consequently, whenever $Z(A_n)$ is nonsingular, 
$$
\val(\Res(P_{A_n},Z))=\val(\det(Z(A_n)))
=\val(\#\Cok(Z(A_n))).
$$
Thus a statistic of the eigenvalues of $A_n$ is transformed into a
statistic of a random cokernel.

The universality theorem of Cheong-Yu
\cite{cheong2023distribution}, whose proof builds on the surjection
moment method of Sawin-Wood \cite{sawin2022moment}, then gives the
required convergence of the resultant distributions for every fixed test
polynomial $Z$.

The second ingredient is control of the degree of the distinguished
factor. For the present paper, ordinary tightness is not enough, since our
main theorem concerns expectations of arbitrary fixed-order joint
eigenvalue statistics. We prove the stronger estimate
\begin{equation}
\label{eq: matrix moment tightness intro}
\lim_{D\to\infty}\sup_{n\geq 1}
\E\left[\left(\deg P_{A_n}^S
\right)^{m_0}\mathbf{1}_{\{\deg P_{A_n}^S>D\}}\right]=0,
\qquad m_0\geq 1.
\end{equation}
The proof follows the spirit of the surjection-moment method developed in the study of random cokernels, where one typically studies
$$
\E\#\Sur(\cdot,G)
$$
for a fixed finite target $G$. In our setting, however, the target must be allowed to grow with the matrix size. Let $B_n\in\Mat_n(\F_p)$ such that
$$
B_n:=A_n\bmod p,
$$
and
$$
M_n:=\Cok_{\F_p[t]}(tI_n-B_n).
$$
We establish estimates for
$$
\E\#\Sur_{\F_p[t]}(M_n,G_n)
$$
that are uniform for finite $\F_p[t]$-modules $G_n$ satisfying
$$
\dim_{\F_p}G_n=O(\log n).
$$
This additional uniformity controls both the largest part and the number of parts of the primary partitions of $B_n$, from which
\eqref{eq: matrix moment tightness intro} follows.

Finally, in the additive Haar ensemble, the limiting distinguished
polynomial is almost surely squarefree. After conditioning on a fixed
primary degree, its law reduces to the characteristic polynomial of a
fixed-dimensional Haar matrix subject to the corresponding primary
condition. This is the same finite-dimensional reduction underlying
\cite[Theorem~4.1]{ellenberg2011modeling}. The nonsquarefree locus is
contained in the zero set of the discriminant and therefore has Haar
measure zero.

These three ingredients verify the hypotheses of the third part of
\Cref{thm: resultant distribution method introduction} and prove
\Cref{thm: random matrix eigenvalue universality}. An intermediate
consequence, which is useful in its own right, is that for every finite
$S\subseteq\Z_{>0}$ there exists a universal random polynomial $Q_\infty^S\in\mathscr{P}^S$ such that
$$
P_{A_n}^S\overset{d}{\longrightarrow}
Q_\infty^S.
$$
The limiting law agrees with the corresponding additive Haar matrix
limit and is independent of the distributions of the entries of $A_n$.

\subsection{Random polynomials with independent coefficients}

The same resultant distribution method also applies to ordinary random
polynomials with independent coefficients. We regard this as a second
application of the general framework.

Let
\begin{equation}\label{item: P_n}
P_n(x)=\xi_nx^n+\cdots+\xi_1x+\xi_0
\in\Z_p[x],
\end{equation}
where the coefficients are independent and $\epsilon$-balanced.
Questions about universality of roots of random polynomials have a long
history over $\R$. For example, Kac
\cite{kac1943average} computed the expected number of real roots in the
Gaussian model. Erd\H{o}s-Offord
\cite{erdos1956number} and Ibragimov-Maslova \cite{ibragimov1971expected} developed early universality results for the expected number of real roots under much more general coefficient
distributions. More recently, Tao-Vu
\cite{tao2015local} established local universality for zeros of random
polynomials, and Nguyen-Vu \cite{nguyen2022roots} developed a general
framework for such local universality phenomena.

Over $\Q_p$, Shmueli \cite{shmueli2023expected} proved that for
independent $\epsilon$-balanced coefficients \footnote{The results  of Shmueli \cite{shmueli2023expected} are formulated for monic $p$-adic random polynomials. The same argument applies, with only minor modifications, when the leading coefficient is random and satisfies the same non-concentration assumption as the remaining coefficients.},
$$
\lim_{n\to\infty}
\E\left[Z_{\Z_p^\times}(P_n)
\right]=\frac{p-1}{p+1}.
$$
For additive Haar coefficients, much more general joint root statistics
were obtained by Caruso \cite{caruso2022zeroes}. We briefly translate his
notation into the one used in the present paper. Fix finite extensions
$K_1,\ldots,K_m$ of $\Q_p$, set
$$
E:=K_1\times\cdots\times K_m,
\qquad
r:=\sum_{i=1}^m[K_i:\Q_p],
$$
and let
\begin{equation}\label{eq: Haar random polynomial}
P_n^{\Haar}(x)=\eta_nx^n+\cdots+\eta_1x+\eta_0,
\end{equation}
where the coefficients are independent and additive Haar distributed on
$\Z_p$. In \cite[Theorem~5.8]{caruso2022zeroes}, Caruso denotes by
$\rho_{E,n}$ the corresponding correlation density on the product
$K_1\times\cdots\times K_m$, and by $Z_{U,n}^{\new}$ the number of ordered
root configurations in $U$ whose coordinates generate the prescribed
extensions and lie in pairwise distinct Galois orbits. Thus, in the notation
of the present paper,
$$
Z_{U,n}^{\new}=Z_U(P_n^{\Haar}).
$$
Moreover, Caruso \cite[Theorem~5.8]{caruso2022zeroes} proves that
$$
\E\left[Z_U(P_n^{\Haar})\right]=\int_U\rho_{E,n}(x_1,\ldots,x_m)\,dx_1\cdots dx_m
$$
and that $\rho_{E,n}$ is independent of $n$ once $n\geq 2r-1$.

To distinguish this stabilized Haar polynomial correlation function from
the random matrix correlation function
$\rho_{K_1,\ldots,K_m}^{(\infty)}$ used in
\Cref{thm: random matrix eigenvalue universality}, we introduce the notation
$$
\rho_{K_1,\ldots,K_m}^{(\infty),\poly}
:=\rho_{E,n},\qquad n\geq 2r-1.
$$
Thus
$$
\rho_{K_1,\ldots,K_m}^{(\infty),\poly}:\mc{O}_{K_1}^{\new}\times\cdots\times\mc{O}_{K_m}^{\new}\longrightarrow\R_{\geq 0}
$$
is precisely Caruso's stabilized correlation density, rewritten in terms
of the tuple of extensions $K_1,\ldots,K_m$ rather than the product $E$.
It describes the stabilized local density of ordered root configurations
for the additive Haar coefficient model.

In the random polynomial setting, universality under the sole
$\epsilon$-balanced assumption must be restricted to roots of absolute
value $1$. Accordingly, define
$$
\mc{O}_{K}^{\times,\new}
:=
\mc{O}_{K}^{\times}
\cap
\mc{O}_{K}^{\new}.
$$
For a nonzero polynomial $P\in\Z_p[x]$, we retain the notation $Z_U(P)$
for the corresponding root-tuple count. Since the event $P_n=0$ may have
positive probability for general discrete coefficient distributions, we
adopt the convention $Z_U(0):=0$. This removes the degenerate case in which
every element of $U$ is formally a root and the usual finite root-counting
statistic is not defined. The same convention is standard in the classical
theory of random polynomials; see, for example, the works of
Kabluchko-Zaporozhets and Nguyen-Vu
\cite{kabluchko2014asymptotic,nguyen2022roots}.

Our random polynomial universality theorem is the following.

\begin{thm}[Universality for random polynomials]
\label{thm: root correlation universality}
Let $K_1,\ldots,K_m$ be finite extensions of $\Q_p$, and let
$$
U\subseteq
\mc{O}_{K_1}^{\times,\new}
\times\cdots\times
\mc{O}_{K_m}^{\times,\new}
$$
be clopen. Suppose that
$$
P_n(x)=\xi_nx^n+\cdots+\xi_1x+\xi_0,
$$
where $\xi_0,\ldots,\xi_n$ are independent
$\epsilon$-balanced random variables in $\Z_p$. Then
$$
\lim_{n\to\infty}
\E\left[
Z_U(P_n)
\right]
=
\int_U
\rho_{K_1,\ldots,K_m}^{(\infty),\poly}
(x_1,\ldots,x_m)\,dx_1\cdots dx_m.
$$
Here, the integral is taken with respect to the product Haar probability
measure on
$\mc{O}_{K_1}\times\cdots\times\mc{O}_{K_m}$, and
$\rho_{K_1,\ldots,K_m}^{(\infty),\poly}$ is the stabilized Haar polynomial
correlation function introduced above.
\end{thm}

It is also worth mentioning the related work of He-Pham-Xu
\cite{he2023universality}, who proved universality results for low-degree
factors of random polynomials over the finite field $\F_p$. From this
perspective, \Cref{thm: root correlation universality} may be viewed as a
$p$-adic refinement of their finite-field universality phenomenon. 

The restriction to unit roots in \Cref{thm: root correlation universality} is necessary without further assumptions on
the coefficient distributions. For example, suppose that all the
coefficients lie in $\Z_p^\times$ almost surely. Dividing by the leading
coefficient gives a monic polynomial over $\Z_p$, so all of its roots are
algebraic integers. Since its constant coefficient is also a unit, the
product of the absolute values of its roots is $1$. Hence every root has
absolute value exactly $1$, and there are no roots in the open unit disk.
The Haar coefficient model, on the other hand, has roots in the open unit
disk with positive probability. Thus statistics away from the unit locus
can retain information about the coefficient distribution.

When $m=1$, $K_1=\Q_p$, and $U=\Z_p^\times$,
\Cref{thm: root correlation universality} recovers the universality
theorem of Shmueli
\cite[Theorem~1]{shmueli2023expected}. More generally, it extends
Caruso's Haar-model correlation formulas  \cite[Theorem 5.8]{caruso2022zeroes} to arbitrary independent $\epsilon$-balanced coefficients and to joint configurations of roots in finite extensions of $\Q_p$.

The proof of \Cref{thm: root correlation universality} again proceeds
through resultant distributions. For test polynomials whose constant terms
are $p$-adic units, we prove convergence of resultant valuations by a
Fourier-analytic argument building on Breuillard-Varj\'u
\cite{breuillard2019irreducibility}. A direct divisibility argument gives
the required moment tightness for the corresponding distinguished factors.
The resultant distribution method then yields convergence of the expected
root statistics.

The random matrix and random polynomial applications are therefore
parallel, but the inputs are rather different. For independent-coefficient
polynomials, the resultant becomes a sum of independent random variables
after reduction modulo the test polynomial, making Fourier analysis
effective. For characteristic polynomials of random matrices, the
coefficients are strongly dependent; instead,
\eqref{eq: intro resultant matrix identity} transforms the resultant into
the determinant of a polynomial evaluation of the matrix, bringing random
cokernels and surjection moments into the problem. The resultant
distribution method provides a common framework in which these two
different mechanisms lead to analogous universality conclusions.

\subsection{Further directions for universality}

We conclude by mentioning two directions in which the universality results
above may be extended. The first is to random matrix ensembles with
additional symmetry, such as symmetric or Hermitian matrices. On the cokernel side, universality is already known for several structured ensembles. Wood \cite{wood2017distribution} developed the surjection-moment method in the study of random symmetric matrices and sandpile groups; Nguyen-Wood \cite{nguyen2025local} established local and global cokernel
universality for random symmetric, skew-symmetric, and Laplacian matrices;
and Lee \cite{lee2023universality} proved the corresponding universality for
random $p$-adic Hermitian matrices. Sawin-Wood \cite{sawin2022moment} subsequently placed the surjection-moment method in a
general categorical framework. From the perspective of the resultant distribution
method, the main missing input is a symmetry-adapted analogue of the theorem
of Cheong-Yu \cite{cheong2023distribution}. Indeed, the identity
$$
\Res(P_A,Z)=\det(Z(A))
$$
reduces the problem to understanding the cokernel of $Z(A_n)$, rather than
only that of $A_n$. Once such a universality theorem and the corresponding
moment-tightness estimates are available, the present argument should in
principle apply without requiring an explicit computation of the limiting
Haar correlation functions.

A second direction is to allow the balancedness parameter to decay with the
matrix size or polynomial degree. For random matrices, one may take
$\epsilon=\epsilon_n\to0$ and ask for the sparsest regime in which the Haar
eigenvalue statistics remain universal. The analogous problem for cokernels
is much better understood: Nguyen-Wood \cite{nguyen2022random} established
universality in sparse regimes, while Jung-Lee-Yu
\cite{jung2026sharp} recently identified the sharp threshold over $\Z_p$ at
the scale
$$
\epsilon_n\asymp \frac{\log n}{n}.
$$
It is natural to ask whether eigenvalue universality exhibits the same
threshold or a genuinely different one. From the viewpoint of our method,
this would require a version of the Cheong-Yu theorem that remains uniform
as $\epsilon_n\to0$, together with sufficiently uniform surjection-moment
estimates to retain moment tightness.

There is a parallel question for random polynomials. The fixed-$\epsilon$
assumption in \Cref{thm: root correlation universality} is not intrinsic to
the Fourier-analytic argument, and with quantitative bookkeeping the same
method should extend to coefficient distributions with
$\epsilon=\epsilon_n\to0$, provided that $\epsilon_n$ decays sufficiently
slowly. Determining the optimal sparse regime, and the corresponding rates
of convergence, remains an interesting problem in both the matrix and
polynomial settings.

\subsection{Outline of the paper}

The remainder of the paper is organized as follows.

In \Cref{sec: Preliminaries}, we collect the basic algebraic and
measure-theoretic background used throughout the paper. We fix our
conventions for finite extensions of $\Q_p$ and their Haar probability
measures, recall the limiting eigenvalue correlation functions for additive
Haar random matrices from \cite{shen2026eigenvalues}, and review Hensel
factorization and the basic properties of the resultant.

In \Cref{sec: resultant distribution method}, we develop the resultant
distribution method in an abstract setting. We first show that the law of a
random $p$-adic polynomial is determined by the distributions of its
resultant valuations against fixed test polynomials. We then prove that
convergence of these resultant distributions, together with degree
tightness, yields weak convergence of the underlying random polynomials.
Finally, under moment tightness and almost-sure squarefreeness of the
limiting polynomial, we obtain convergence of the corresponding root
statistics and their expectations.

In \Cref{sec: The random matrix case}, we apply this framework to
characteristic polynomials of random matrices. The identity
$$
\Res(P_{A_n},Z)=\det(Z(A_n))
$$
converts resultant valuations into cokernel statistics of polynomial
evaluations of $A_n$. The universality theorem of Cheong-Yu then gives
convergence of the resultant distributions. We compare the resulting limit
with the additive Haar ensemble, establish the required squarefreeness of
the Haar limiting polynomial, and, modulo the moment-tightness estimate
proved in the following section, deduce
\Cref{thm: random matrix eigenvalue universality}. We also derive the
concrete consequences stated in
\Cref{thm: Zp first second moments universality} and \Cref{thm: quadratic extension universality}.

\Cref{sec: Moment Tightness} is devoted to the remaining moment-tightness
input for the random matrix argument. We prove a growing-target version of
the surjection-moment estimate for the random
$\F_p[t]$-module associated with $A_n\bmod p$, allowing the target dimension
to grow logarithmically with $n$. Fourier and combinatorial estimates then
yield quantitative tail bounds for the primary partitions of
$A_n\bmod p$, from which we deduce moment tightness for all fixed
distinguished degrees.

Finally, in \Cref{sec: random polynomial}, we apply the resultant
distribution method to random polynomials with independent
$\epsilon$-balanced coefficients. For test polynomials with unit constant
term, a Fourier-analytic argument shows that the resultant valuations
converge to the same limits as in the additive Haar coefficient model.
A direct divisibility argument supplies the corresponding moment-tightness
estimate. Combining these inputs with the general theory from
\Cref{sec: resultant distribution method} proves
\Cref{thm: root correlation universality}.

\subsection{Statement on the use of AI}

The main ideas and proof strategy of this paper were developed by the author,
and a complete draft was written before artificial intelligence tools were
used. AI tools were subsequently used to assist with a systematic review of
the manuscript, including organizational and editorial improvements,
language polishing, and checks of technical details. In particular, they
were used to help review the proof of moment tightness in \Cref{sec: Moment Tightness} and to
identify places where intermediate estimates or quantifiers could be stated
more explicitly. All mathematical claims and arguments were independently
checked by the author, who takes full responsibility for the accuracy and
content of the paper.

%% file: Preliminaries.tex
\section{Preliminaries}
\label{sec: Preliminaries}

In this section, we collect the basic notation and preliminary results used
throughout the paper. We first fix our conventions for finite extensions of
$\Q_p$, their rings of integers, and the associated Haar probability measures.
We then recall the limiting eigenvalue correlation functions for additive Haar
random matrices obtained in \cite{shen2026eigenvalues}. Finally, we record the
form of Hensel's lemma and the elementary properties of the resultant that
will be used repeatedly below.

\subsection{Finite extensions and Haar measure}

Throughout the paper, we fix a prime $p$ and an algebraic closure
$\overline{\Q}_p$ of $\Q_p$. Every finite extension of $\Q_p$ will be viewed
as a subfield of $\overline{\Q}_p$. We denote by
$$
\val:\overline{\Q}_p\longrightarrow\Q\cup\{\infty\}
$$
the extension of the $p$-adic valuation normalized by $\val(p)=1$, and write
$$
\lVert x\rVert:=p^{-\val(x)}.
$$

Let $K/\Q_p$ be a finite extension. We denote by $\mc{O}_K$ its ring of
integers and by $\mc{O}_K^\times$ its group of units. We further define
$$
K^{\new}
:=
\left\{
x\in K:\Q_p[x]=K
\right\},
$$
and set
$$
\mc{O}_K^{\new}
:=
K^{\new}\cap\mc{O}_K,
\qquad
\mc{O}_K^{\times,\new}
:=
K^{\new}\cap\mc{O}_K^\times.
$$
Thus $\mc{O}_K^{\new}$ consists of the algebraic integers that generate
$K$ over $\Q_p$.

The compact additive group $\mc{O}_K$ carries a unique Haar probability
measure, which we denote by $\mu_K$. Thus
$$
\mu_K(\mc{O}_K)=1.
$$
If $\pi_K$ is a uniformizer of $K$ and the residue field of $K$ has
cardinality $q_K$, then
$$
\mu_K(a+\pi_K^r\mc{O}_K)=q_K^{-r}
$$
for every $a\in\mc{O}_K$ and $r\geq 0$.

The complement of $\mc{O}_K^{\new}$ in $\mc{O}_K$ has Haar measure zero.
Indeed,
$$
\mc{O}_K\setminus\mc{O}_K^{\new}
=
\bigcup_{\Q_p\subseteq L\subsetneq K}\mc{O}_L,
$$
where the union is over the proper intermediate fields of $K/\Q_p$.
There are only finitely many such fields, and each $\mc{O}_L$ lies in a
proper $\Q_p$-linear subspace of $K$. Hence
$$
\mu_K(\mc{O}_K^{\new})=1.
$$

For finite extensions $K_1,\ldots,K_m$ of $\Q_p$, we equip
$$
\mc{O}_{K_1}\times\cdots\times\mc{O}_{K_m}
$$
with the product Haar probability measure
$$
\mu_{K_1}\otimes\cdots\otimes\mu_{K_m}.
$$
Accordingly, an integral written as
$$
\int_U
f(x_1,\ldots,x_m)\,dx_1\cdots dx_m
$$
is always understood with respect to this product measure.

Recall also that, for
$$
U\subseteq
\mc{O}_{K_1}^{\new}
\times\cdots\times
\mc{O}_{K_m}^{\new},
$$
the statistic $Z_U(P)$ counts tuples
$(x_1,\ldots,x_m)\in U$ of roots of $P$ lying in pairwise distinct
Galois orbits over $\Q_p$. Thus the notation records both the locations of
the roots and the finite extensions of $\Q_p$ that they generate.

\subsection{Haar random matrices and eigenvalue correlations}
\label{subsec: Haar random matrix correlations}

We next recall the limiting eigenvalue correlation functions for additive
Haar random matrices obtained in \cite{shen2026eigenvalues}. In this
subsection, we record only those results from \cite{shen2026eigenvalues}
that will be used in the present paper. Their derivation will be treated
entirely as a black box here: we do not repeat the finite-dimensional
correlation formulas or the arguments leading to their limiting form.
Readers interested in the construction and explicit analysis of these
correlation functions are referred to \cite{shen2026eigenvalues}.

For every
$n\geq 1$, let
$$
A_n^{\Haar}\in\Mat_n(\Z_p)
$$
be a random matrix whose entries are independent and distributed according
to additive Haar probability measure on $\Z_p$, and write
$$
P_{A_n^{\Haar}}(x)
:=
\det(xI_n-A_n^{\Haar})
$$
for its characteristic polynomial.

The following result is the only input from the explicit Haar theory that
we will need in the general universality argument.

\begin{thm}[Limiting eigenvalue correlations for Haar random matrices]
\label{thm: Haar random matrix correlations}
Let $K_1,\ldots,K_m$ be finite extensions of $\Q_p$. There exists a
nonnegative function
$$
\rho_{K_1,\ldots,K_m}^{(\infty)}:\mc{O}_{K_1}^{\new}\times\cdots\times\mc{O}_{K_m}^{\new}\longrightarrow\R_{\geq 0}
$$
such that, for every measurable set
$$
U\subseteq
\mc{O}_{K_1}^{\new}
\times\cdots\times
\mc{O}_{K_m}^{\new},
$$
we have
$$
\lim_{n\to\infty}\E\left[Z_U(P_{A_n^{\Haar}})\right]=\int_U\rho_{K_1,\ldots,K_m}^{(\infty)}(x_1,\ldots,x_m)\,dx_1\cdots dx_m.
$$
\end{thm}

\begin{proof}
This is a direct reformulation of
\cite[Theorem~6.2]{shen2026eigenvalues}. We explain the translation of
notation.

In \cite{shen2026eigenvalues}, for fixed finite extensions
$K_1,\ldots,K_m$ of $\Q_p$, the product
$$
K_1\times\cdots\times K_m
$$
is denoted by $E$, and the limiting correlation function is written as
$\rho_E^{(\infty)}$. As noted there, this is the same function that we
denote in the present paper by $\rho_{K_1,\ldots,K_m}^{(\infty)}$.

Moreover, the random variable $Z_{U,n}$ appearing in
\cite[Theorem~6.2]{shen2026eigenvalues} counts tuples
$(x_1,\ldots,x_m)\in U$ of eigenvalues of an $n\times n$ additive Haar
random matrix, with the coordinates lying in pairwise distinct Galois
orbits. Since the eigenvalues of $A_n^{\Haar}$ are precisely the roots of
its characteristic polynomial
$$
P_{A_n^{\Haar}}(x)=\det(xI_n-A_n^{\Haar}),
$$
this is exactly the random variable $Z_U(P_{A_n^{\Haar}})$ in our notation.

Finally, the integration in \cite[Theorem~6.2]{shen2026eigenvalues} is
with respect to the product of the additive Haar probability measures on
$\mc{O}_{K_1},\ldots,\mc{O}_{K_m}$, which agrees with our convention.
Therefore that theorem gives
$$
\lim_{n\to\infty}
\E\left[
Z_U(P_{A_n^{\Haar}})
\right]
=
\int_U
\rho_{K_1,\ldots,K_m}^{(\infty)}
(x_1,\ldots,x_m)
\,dx_1\cdots dx_m,
$$
as claimed.
\end{proof}

The function
$\rho_{K_1,\ldots,K_m}^{(\infty)}$ will serve as the reference limiting
eigenvalue statistic throughout the paper. In \cite[Section~9]{shen2026eigenvalues}, several of these limiting
correlation functions were evaluated explicitly in low-degree cases.
The present paper does not require their general explicit form. Our goal is instead to show that the same functions govern
the limiting eigenvalue statistics of random matrices with general
independent $\epsilon$-balanced entries.

\subsection{Hensel's lemma and the resultant}
\label{subsec:zpbar}

We now recall the algebraic facts about polynomial factorization and
resultants that will be used throughout the paper.

The following is the strong form of Hensel's lemma.

\begin{lemma}[Hensel's lemma]
\label{lem: Hensel}
Let $Z(x)\in\Z_p[x]$, and let $\overline{Z}(x)\in\F_p[x]$ denote its
reduction modulo $p$. Suppose that
$$
\overline{Z}(x)
=
\overline{Z}_1(x)\overline{Z}_2(x),
$$
where $\overline{Z}_1,\overline{Z}_2\in\F_p[x]$ are relatively prime.
Then $Z$ admits a factorization
$$
Z(x)=Z_1(x)Z_2(x)
$$
with $Z_1,Z_2\in\Z_p[x]$ whose reductions modulo $p$ are
$\overline{Z}_1$ and $\overline{Z}_2$, respectively.
\end{lemma}

\begin{proof}
See \cite[Theorem~II.4.6]{neukirch2013algebraic}.
\end{proof}

Let
$$
Z_1(x)
=
a_{r_1}x^{r_1}+\cdots+a_0,
\qquad
Z_2(x)
=
b_{r_2}x^{r_2}+\cdots+b_0
$$
be nonzero polynomials in $\Z_p[x]$ of degrees $r_1$ and $r_2$,
respectively. Let
$x_{i,1},\ldots,x_{i,r_i}\in\overline{\Q}_p$ denote the roots of $Z_i$,
counted with multiplicity. Their \emph{resultant} is
$$
\Res(Z_1,Z_2)=a_{r_1}^{r_2}b_{r_2}^{r_1}
\prod_{l_1=1}^{r_1}
\prod_{l_2=1}^{r_2}
(x_{1,l_1}-x_{2,l_2})
\in\Z_p.
$$
We also set $\Res(Z_1,1)=\Res(1,Z_2)=1$, and
$$\Res(Z_1,0)=\Res(0,Z_2)=0$$
whenever $Z_1,Z_2$ are non-constant.

The resultant records whether two polynomials have a common root and,
more generally, the $p$-adic distances between their roots. It also admits
a useful linear-algebraic interpretation. If $Z_1$ is monic, then
$\Z_p[x]/(Z_1)$ is a finite free $\Z_p$-module and
$$
\Res(Z_1,Z_2)
=
\det\left(
m_{Z_2}:
\Z_p[x]/(Z_1)
\longrightarrow
\Z_p[x]/(Z_1)
\right),
$$
where $m_{Z_2}$ denotes multiplication by $Z_2$.

We record the standard properties of the resultant that will be used below.

\begin{prop}
\label{prop: properties of resultant}
Let $Z_1,Z_2,Z_3\in\Z_p[x]$. Then:
\begin{enumerate}
\item
$\Res(Z_1,Z_2)=(-1)^{\deg Z_1\deg Z_2}\Res(Z_2,Z_1)$.

\item
$\Res(Z_1,Z_2)=0$ if and only if $Z_1$ and $Z_2$ have a common root.

\item
We have
$$
\Res(Z_1Z_2,Z_3)=\Res(Z_1,Z_3)\Res(Z_2,Z_3),
$$
and
$$
\Res(Z_1,Z_2Z_3)=\Res(Z_1,Z_2)\Res(Z_1,Z_3).
$$

\item
If $Z_1$ is monic, then
$$
\Res(Z_1,Z_3)
=
\Res(Z_1,Z_3-Z_1Z_2).
$$
\end{enumerate}
\end{prop}

The following elementary consequence will be particularly useful when we
separate roots lying in different lifted subspaces.

\begin{prop}
\label{prop: relatively prime resultant}
Suppose that $Z_1,Z_2\in\Z_p[x]$, at least one of which is monic. If their
reductions
$\overline{Z}_1,\overline{Z}_2\in\F_p[x]$ are relatively prime, then
$$
\val\left(\Res(Z_1,Z_2)\right)=0.
$$
\end{prop}

\begin{proof}
Suppose first that $Z_1$ is monic, and let $d:=\deg Z_1$. Multiplication by
$Z_2$ defines a $\Z_p$-linear map
$$
m_{Z_2}:
\Z_p[x]/(Z_1)
\longrightarrow
\Z_p[x]/(Z_1),
$$
and
$$
\Res(Z_1,Z_2)=\det(m_{Z_2}).
$$
After reduction modulo $p$, this becomes multiplication by
$\overline{Z}_2$ on $\F_p[x]/(\overline{Z}_1)$. Since
$\overline{Z}_1$ and $\overline{Z}_2$ are relatively prime, the class of
$\overline{Z}_2$ is invertible in this quotient. Hence the reduced map is
invertible and
$$
\det(m_{Z_2})\not\equiv 0\pmod p.
$$
Therefore $\Res(Z_1,Z_2)\in\Z_p^\times$, proving the claim.

If instead $Z_2$ is monic, the result follows by symmetry from
\Cref{prop: properties of resultant}.
\end{proof}

%% file: The_resultant_distribution_method.tex
\section{The resultant distribution method} \label{sec: resultant distribution method} 
In this section, we develop the resultant distribution method introduced in
\Cref{sec: introduction}. Its purpose is to recover the law, or the limiting
law, of a random $p$-adic polynomial from the distributions of its resultants
against a sufficiently rich family of fixed test polynomials. The arguments in this section are formulated at an abstract level and do not depend on the specific random polynomial or random matrix models considered later.

Throughout this section, we fix a finite set
$S\subset\Z_{>0}$, and retain the notation $\mathscr{P}^S$ introduced in \Cref{sec: introduction} for the space of monic polynomials over $\Z_p$ whose roots lie in the lifted subspaces indexed by $S$. For every $d\geq 0$, let $\mathscr{P}_d^S$ denote the subset consisting of polynomials of degree $d$, and equip
$$
\mathscr{P}^S=\bigsqcup_{d\geq 0}\mathscr{P}_d^S
$$
with the disjoint-union topology induced by the natural coefficient topology
on each fixed-degree component.

The goal of this section is to establish the three assertions stated in
\Cref{thm: resultant distribution method introduction} in precise form. We
first prove that the law of a random polynomial in $\mathscr{P}^S$ is
determined by its resultant-valuation distributions. We then show that
convergence of these distributions, together with tightness of the degrees,
implies weak convergence of the random polynomials. Finally, under a stronger
moment-tightness assumption, we upgrade this weak convergence to convergence
of expected root statistics. Accordingly, we formulate and prove these three assertions as
\Cref{thm: resultant distribution determination introduction},
\Cref{thm: resultant distribution convergence introduction}, and
\Cref{thm: expected root statistics convergence}, respectively. We begin with
the determination theorem.

\begin{thm}[Determination by resultant distributions]
\label{thm: resultant distribution determination introduction}
Let $P^S_1$ and $P^S_2$ be random elements of $\mathscr{P}^S$. Suppose that, for every fixed polynomial $Z\in\mathscr{P}^S$, we have
$$
\val\bigl(\Res(P^S_1,Z)\bigr)\overset{d}{=}\val\bigl(\Res(P^S_2,Z)\bigr).
$$
Then $P^S_1\overset{d}{=}P^S_2$ as random elements of $\mathscr{P}^S$.
\end{thm}

The following theorem is the convergence counterpart of
\Cref{thm: resultant distribution determination introduction}. It shows that
convergence of all resultant distributions, together with tightness of the
degrees, yields convergence of the random polynomials at every finite
coefficient level.

\begin{thm}[Convergence from resultant distributions]
\label{thm: resultant distribution convergence introduction}
Let $P^S_1,P^S_2,\ldots$ be random elements of $\mathscr{P}^S$. Assume that the
following conditions hold:
\begin{enumerate}
\item For every fixed polynomial $Z\in\mathscr{P}^S$, the sequence of random
variables
$$
\val\bigl(\Res(P^S_n,Z)\bigr),\qquad n\geq 1,
$$
converges weakly in $\Z_{\geq 0}\cup\{\infty\}$ as $n\to\infty$. \label{item: resultant convergence}

\item The sequence $\{\deg P^S_n\}_{n\geq 1}$ is tight; namely, 
$$
\lim_{D\to\infty}\sup_{n\geq 1}
\mathbf{P}(\deg P^S_n>D)=0.
$$ \label{item: degree tightness}
\end{enumerate}
Then there exists a random element $P^S_\infty$ of $\mathscr{P}^S$, unique in
distribution, such that
$$
P^S_n
\overset{d}{\longrightarrow}
P^S_\infty,
\qquad n\to\infty.
$$
Furthermore, for every fixed polynomial $Z\in\mathscr{P}^S$, we have
$$
\val\bigl(\Res(P^S_n,Z)\bigr)
\overset{d}{\longrightarrow}
\val\bigl(\Res(P^S_\infty,Z)\bigr),
\qquad n\to\infty.
$$
\end{thm}

The degree-tightness assumption in
\Cref{thm: resultant distribution convergence introduction} is sufficient
to obtain weak convergence of the random polynomials, but it does not by
itself control the contribution of rare polynomials of large degree to the
expected root statistics. Since our main theorem concerns convergence of
expected $m$-tuple counts, we need the stronger moment tightness condition
below, which ensures uniform integrability of the corresponding root-counting random variables. The following result is therefore the form of the convergence theorem that will be applied in the proof of our main
universality theorem.

\begin{thm}[Convergence of root statistics and their expectations]
\label{thm: expected root statistics convergence}
Fix an integer $m\geq 1$. Let $K_1,\ldots,K_m$ be finite extensions
of $\Q_p$, and let
$$
U\subseteq\mc{O}_{K_1}^{\new}
\times\cdots\times\mc{O}_{K_m}^{\new}
$$
be clopen. Let $P_1^S,P_2^S,\ldots$ be random elements of $\mathscr{P}^S$
satisfying the assumptions of
\Cref{thm: resultant distribution convergence introduction}, and let
$P_\infty^S\in\mathscr{P}^S$ be the limiting random polynomial provided
by that theorem. Suppose that
$$
\mathbf{P}\left(P_\infty^S\text{ has a repeated root}\right)=0.
$$
Then $Z_U(P_n^S)\overset{d}{\longrightarrow}
Z_U(P_\infty^S)$. If, in addition, the degree tightness assumption is strengthened to moment tightness, i.e.,
$$
\lim_{D\to\infty}\sup_{n\geq 1}
\mathbf{E}\left[(\deg P_n^S)^m
\mathbf{1}_{\{\deg P_n^S>D\}}
\right]=0,
$$
then
$$
\lim_{n\to\infty}\mathbf{E}\left[
Z_U(P_n^S)\right]=\mathbf{E}\left[
Z_U(P_\infty^S)\right].
$$
\end{thm}

\subsection{Determination by resultant distributions}
\label{subsec: determination by resultant distributions}

In this subsection, we prove
\Cref{thm: resultant distribution determination introduction}. 
For every $d\geq 0$, let $\mathscr{P}_d^S$ denote the subset of
$\mathscr{P}^S$ consisting of polynomials of degree $d$, and define
$$
\mathscr{P}_{\leq d}^S:=\bigcup_{0\leq e\leq d}\mathscr{P}_e^S.
$$
Under the natural coefficient identification, the space
$\mathscr{P}_d^S$ is a Borel subset of $\Z_p^d$. We equip
$$
\mathscr{P}^S=\bigsqcup_{d\geq 0}\mathscr{P}_d^S
$$
with the corresponding disjoint-union Borel structure.

The proof consists of two steps. We first show that the assumed
one-dimensional resultant distributions determine the joint distributions of
any finite collection of resultant valuations. We then show that a sufficiently
rich collection of such valuations determines the polynomial itself. 

\begin{lemma}
\label{lem: joint resultant valuation determination}
Let $P^S_1$ and $P^S_2$ be random elements of $\mathscr{P}^S$. Suppose that, for
every fixed $Z\in\mathscr{P}^S$,
$$
\val\bigl(\Res(P^S_1,Z)\bigr)
\overset{d}{=}
\val\bigl(\Res(P^S_2,Z)\bigr).
$$
Then, for every fixed collection $Z_1,\ldots,Z_r\in\mathscr{P}^S$, we have
$$
\left(
\val\bigl(\Res(P^S_1,Z_1)\bigr),
\ldots,
\val\bigl(\Res(P^S_1,Z_r)\bigr)
\right)
\overset{d}{=}
\left(
\val\bigl(\Res(P^S_2,Z_1)\bigr),
\ldots,
\val\bigl(\Res(P^S_2,Z_r)\bigr)
\right).
$$
\end{lemma}

\begin{proof}
For $j\in\{1,2\}$ and $1\leq i\leq r$, define
$$
Y_{j,i}:=\left\|\Res(P^S_j,Z_i)\right\|\in[0,1].
$$
We will show that the random vectors
$$
(Y_{1,1},\ldots,Y_{1,r})
\quad\text{and}\quad
(Y_{2,1},\ldots,Y_{2,r})
$$
have the same mixed moments.

Fix $(m_1,\ldots,m_r)\in\Z_{\geq 0}^r$. The polynomial $Z_1^{m_1}\cdots Z_r^{m_r}$ belongs to $\mathscr{P}^S$. By the multiplicativity of the resultant in
\Cref{prop: properties of resultant}, we have
$$
\left\|
\Res\left(P^S_j,Z_1^{m_1}\cdots Z_r^{m_r}\right)
\right\|=
\prod_{i=1}^r Y_{j,i}^{m_i},
\qquad j\in\{1,2\},
$$
where we use the convention $y^0=1$ for every $y\in[0,1]$.

By assumption,
$$
\val\left(
\Res\left(P^S_1,Z_1^{m_1}\cdots Z_r^{m_r}\right)
\right)
\overset{d}{=}
\val\left(
\Res\left(P^S_2,Z_1^{m_1}\cdots Z_r^{m_r}\right)
\right).
$$
Applying the map $a\mapsto p^{-a}$, with $p^{-\infty}:=0$, gives
$$
\left\|
\Res\left(P^S_1,Z_1^{m_1}\cdots Z_r^{m_r}\right)
\right\|
\overset{d}{=}
\left\|
\Res\left(P^S_2,Z_1^{m_1}\cdots Z_r^{m_r}\right)
\right\|.
$$
Consequently,
$$
\E\left[\prod_{i=1}^r Y_{1,i}^{m_i}\right]
=
\E\left[\prod_{i=1}^r Y_{2,i}^{m_i}\right].
$$
Thus the two $[0,1]^r$-valued random vectors have the same mixed
polynomial moments.

By the Stone-Weierstrass theorem, polynomials in the coordinate
functions are dense in $C([0,1]^r)$. It follows that the two random
vectors have the same integrals against every continuous function on
$[0,1]^r$, and hence
$$
(Y_{1,1},\ldots,Y_{1,r})
\overset{d}{=}
(Y_{2,1},\ldots,Y_{2,r}).
$$
Finally, applying coordinatewise the map
$$
y\longmapsto
\begin{cases}
-\log_p y, & y>0,\\
+\infty, & y=0,
\end{cases}
$$
we obtain
$$
\left(
\val(\Res(P^S_1,Z_1)),\ldots,\val(\Res(P^S_1,Z_r))
\right)
\overset{d}{=}
\left(
\val(\Res(P^S_2,Z_1)),\ldots,\val(\Res(P^S_2,Z_r))
\right),
$$
as desired.
\end{proof}

We next formulate the finite-level separation statement that will allow us to recover the law of the polynomial.

\begin{prop}
\label{prop: finite level resultant separation}
Let $d,k\geq 0$ be nonnegative integers, and let $P^S\in\mathscr{P}_{\leq d}^S$ be fixed. Then there exist finitely many polynomials
$Z_1,\ldots,Z_r\in\mathscr{P}^S$ such that the following holds. For every
$Q^S_1\in\mathscr{P}_{\leq d}^S$ satisfying
$$
Q^S_1\equiv P^S\pmod{p^k}
$$
and every $Q^S_2\in\mathscr{P}_{\leq d}^S$ satisfying
$$
Q^S_2\not\equiv P^S\pmod{p^k},
$$
we have
$$
\left(\val\bigl(\Res(Q^S_1,Z_1)\bigr),
\ldots,\val\bigl(\Res(Q^S_1,Z_r)\bigr)
\right)\neq\left(
\val\bigl(\Res(Q^S_2,Z_1)\bigr),
\ldots,
\val\bigl(\Res(Q^S_2,Z_r)\bigr)
\right).
$$
\end{prop}

\begin{proof}
The case $k=0$ is immediate, so assume that $k\geq 1$. Set
$$
\mathscr{A}_{P,k}:=\left\{
Q^S_1\in\mathscr{P}_{\leq d}^S:
Q^S_1\equiv P^S\pmod{p^k}
\right\}
$$
and
$$
\mathscr{B}_{P,k}:=\left\{
Q^S_2\in\mathscr{P}_{\leq d}^S:
Q^S_2\not\equiv P^S\pmod{p^k}
\right\}.
$$
Since $\mathscr{P}_{\leq d}^S$ is a finite union of compact spaces,
both $\mathscr{A}_{P,k}$ and $\mathscr{B}_{P,k}$ are compact. Hence $\mathscr{A}_{P,k}\times\mathscr{B}_{P,k}$
is compact.

We first show that every pair
$$
(Q^S_1,Q^S_2)\in
\mathscr{A}_{P,k}\times\mathscr{B}_{P,k}
$$
can be separated by a resultant valuation. Since $Q^S_1\not\equiv Q^S_2\pmod{p^k}$, we have $Q^S_1\neq Q^S_2$.

Factor $Q^S_1$ and $Q^S_2$ into monic irreducible polynomials over $\Q_p$.
Since $Q^S_1\neq Q^S_2$, there exists a monic irreducible polynomial
$H\in\Z_p[x]$ whose multiplicities in $Q^S_1$ and $Q^S_2$ are different.
Equivalently, one may choose a root whose minimal polynomial occurs with
different multiplicities in $Q^S_1$ and $Q^S_2$. In particular, when one
polynomial has a root that the other does not, we take $H$ to be the minimal
polynomial of that root.

Write
$$
Q^S_i=H^{m_i}R_i,\qquad i\in\{1,2\},
$$
where $m_1\neq m_2$ and $H\nmid R_i$. For $N\geq 1$, define
$$
Z_N:=H+p^N.
$$
Since $Z_N\equiv H\pmod p$, we have $Z_N\in\mathscr{P}^S$.

Let $h:=\deg H$. Since
$$
\Res(H,Z_N)=\Res(H,H+p^N)=p^{Nh},
$$
we have
$$
\val\bigl(\Res(H^{m_i},Z_N)\bigr)=m_iNh.
$$
Moreover, since $H$ and $R_i$ are relatively prime over $\Q_p$, we have
$$
\Res(R_i,H)\neq 0.
$$
As $Z_N\to H$ coefficientwise as $N\to\infty$, it follows that
$$
\val\bigl(\Res(R_i,Z_N)\bigr)=\val\bigl(\Res(R_i,H)\bigr)
$$
for all sufficiently large $N$. By the third item of \Cref{prop: properties of resultant},
for all sufficiently large $N$,
$$
\val\bigl(\Res(Q^S_i,Z_N)\bigr)
=
m_iNh+\val\bigl(\Res(R_i,H)\bigr).
$$
Since $m_1\neq m_2$, these two quantities are different for all sufficiently
large $N$. Thus there exists $Z\in\mathscr{P}^S$ such that
$$
\val\bigl(\Res(Q^S_1,Z)\bigr)
\neq
\val\bigl(\Res(Q^S_2,Z)\bigr).
$$

Choose an integer $M\geq 1$ such that
$$
\val\bigl(\Res(Q^S_1,Z)\bigr)\wedge M
\neq
\val\bigl(\Res(Q^S_2,Z)\bigr)\wedge M.
$$
The map
$$
Q^S\longmapsto
\val\bigl(\Res(Q^S,Z)\bigr)\wedge M
$$
is locally constant on $\mathscr{P}_{\leq d}^S$, since it depends only on
$\Res(Q^S,Z)\bmod p^M$. Hence there exist neighborhoods
$\mathcal{V}_{Q^S_1}$ of $Q^S_1$ and $\mathcal{V}_{Q^S_2}$ of $Q^S_2$ such that
$$
\val\bigl(\Res(\widetilde Q^S_1,Z)\bigr)
\neq
\val\bigl(\Res(\widetilde Q^S_2,Z)\bigr)
$$
for every
$$
\widetilde Q^S_1\in\mathcal{V}_{Q^S_1},\qquad
\widetilde Q^S_2\in\mathcal{V}_{Q^S_2}.
$$

These product neighborhoods form an open cover of $\mathscr{A}_{P,k}\times\mathscr{B}_{P,k}$. By compactness, there is a finite subcover. Let
$Z_1,\ldots,Z_r\in\mathscr{P}^S$ be the corresponding test polynomials.
Then every pair
$$
(Q^S_1,Q^S_2)\in
\mathscr{A}_{P,k}\times\mathscr{B}_{P,k}
$$
is separated by at least one of the resultant valuations associated with
$Z_1,\ldots,Z_r$. This proves the proposition.
\end{proof}

Assuming \Cref{prop: finite level resultant separation} for the moment, we
complete the proof of the determination theorem.

\begin{proof}[Proof of
\Cref{thm: resultant distribution determination introduction}]
It suffices to prove that, for every $k\geq 1$ and every
$Q^S\in\mathscr{P}^S$,
$$
\mathbf{P}\left(P^S_1\equiv Q^S\pmod{p^k}\right)=
\mathbf{P}\left(P^S_2\equiv Q^S\pmod{p^k}\right).
$$
Suppose, toward a contradiction, that this fails for some $k\geq 1$ and
some $Q\in\mathscr{P}^S$. Set
$$
\varepsilon:=\left|
\mathbf{P}\left(P^S_1\equiv Q^S\pmod{p^k}\right)-\mathbf{P}\left(P^S_2\equiv Q^S\pmod{p^k}\right)
\right|>0.
$$
Choose $d\geq\deg Q$ sufficiently large that
$$
\mathbf{P}(\deg P^S_1>d)
+
\mathbf{P}(\deg P^S_2>d)
<
\varepsilon.
$$

Apply \Cref{prop: finite level resultant separation} to $Q^S$, $d$, and
$k$. We obtain polynomials
$$
Z_1,\ldots,Z_r\in\mathscr{P}^S
$$
such that every polynomial in $\mathscr{P}_{\leq d}^S$ congruent to $Q^S$
modulo $p^k$ has a different resultant-valuation vector from every
polynomial in $\mathscr{P}_{\leq d}^S$ not congruent to $Q^S$ modulo $p^k$.

Define
$$
\Phi(R):=\left(\val\bigl(\Res(R,Z_1)\bigr),
\ldots,\val\bigl(\Res(R,Z_r)\bigr)\right)
$$
and let
$$
\mathcal{A}:=\left\{
\Phi(R):
R\in\mathscr{P}_{\leq d}^S,\ 
R\equiv Q^S\pmod{p^k}
\right\}.
$$
By the separating property of the polynomials $Z_1,\ldots,Z_r$, for every
$R\in\mathscr{P}_{\leq d}^S$ we have
$$
\Phi(R)\in\mathcal{A}
\quad\Longleftrightarrow\quad
R\equiv Q^S\pmod{p^k}.
$$
Consequently, for $i\in\{1,2\}$,
$$
\left|\mathbf{P}\left(P^S_i\equiv Q^S\pmod{p^k}\right)-\mathbf{P}\left(\Phi(P^S_i)\in\mathcal{A}\right)
\right|
\leq
\mathbf{P}(\deg P^S_i>d).
$$

On the other hand, by
\Cref{lem: joint resultant valuation determination}, the random vectors
$\Phi(P^S_1)$ and $\Phi(P^S_2)$ have the same distribution. Hence
$$
\mathbf{P}\left(\Phi(P^S_1)\in\mathcal{A}\right)
=
\mathbf{P}\left(\Phi(P^S_2)\in\mathcal{A}\right).
$$
It follows that
\begin{align*}
\varepsilon
&=\left|\mathbf{P}\left(P^S_1\equiv Q^S\pmod{p^k}\right)-\mathbf{P}\left(P^S_2\equiv Q^S\pmod{p^k}\right)
\right| \\
&\leq\mathbf{P}(\deg P^S_1>d)+\mathbf{P}(\deg P^S_2>d)<\varepsilon,
\end{align*}
which is a contradiction. Therefore, for every $k\geq 1$,
$$
P^S_1\bmod p^k\overset{d}{=}P^S_2\bmod p^k.
$$
Since these finite coefficient quotients generate the Borel
$\sigma$-algebra of $\mathscr{P}^S$, we conclude that
$$
P^S_1\overset{d}{=}P^S_2.
$$
\end{proof}

\subsection{Convergence from resultant distributions}
\label{subsec: convergence from resultant distributions}

In this subsection, we prove
\Cref{thm: resultant distribution convergence introduction} and
\Cref{thm: expected root statistics convergence}. 

\begin{proof}[Proof of
\Cref{thm: resultant distribution convergence introduction}]
Recall that for every $D\geq 0$, the finite union
$$
\mathscr{P}_{\leq D}^S=\bigcup_{0\leq d\leq D}\mathscr{P}_d^S
$$
is compact. Therefore, the degree-tightness \eqref{item: degree tightness} assumption implies that the sequence of laws
$\{\mathcal{L}(P^S_n)\}_{n\geq 1}$ is tight on $\mathscr{P}^S$. By Prokhorov's theorem, every subsequence of $\{P^S_n\}_{n\geq 1}$ admits a further subsequence that converges weakly in $\mathscr{P}^S$. Thus, after
passing to a subsequence, there exists a random element
$P^S_\infty\in\mathscr{P}^S$ such that
$P^S_n\overset{d}{\longrightarrow}P^S_\infty$. Notice that for a fixed polynomial $Z\in\mathscr{P}^S$, the map $$\val\bigl(\Res(\cdot,Z)\bigr):\mathscr{P}^S\longrightarrow\Z_{\geq 0}\cup\{\infty\}$$
is continuous, where $\Z_{\geq 0}\cup\{\infty\}$ is equipped with the one-point compactification topology. Therefore, the continuous mapping theorem gives
$$
\val\bigl(\Res(P^S_n,Z)\bigr)
\overset{d}{\longrightarrow}
\val\bigl(\Res(P^S_\infty,Z)\bigr).
$$

Now suppose that another subsequence converges weakly to a random element
$Q^S_\infty\in\mathscr{P}^S$. Applying the same argument to this subsequence,
we obtain, for every fixed $Z\in\mathscr{P}^S$,
$$
\val\bigl(\Res(P^S_\infty,Z)\bigr)
\overset{d}{=}
\val\bigl(\Res(Q^S_\infty,Z)\bigr).
$$
By
\Cref{thm: resultant distribution determination introduction}, it follows
that $P^S_\infty\overset{d}{=}Q^S_\infty$. Thus all subsequential weak limits have the same distribution.

Since the sequence $\{\mathcal{L}(P^S_n)\}_{n\geq 1}$ is tight and has a
unique possible subsequential weak limit, the full sequence converges to $P^S_\infty$ as random elements of $\mathscr{P}^S$. Finally, for every fixed polynomial $Z\in\mathscr{P}^S$, the continuity of the map $P\mapsto\val(\Res(P,Z))$ and the continuous mapping theorem yield
$$
\val\bigl(\Res(P^S_n,Z)\bigr)
\overset{d}{\longrightarrow}
\val\bigl(\Res(P^S_\infty,Z)\bigr).
$$
The uniqueness in distribution of $P^S_\infty$ follows once more from
\Cref{thm: resultant distribution determination introduction}.
\end{proof}

\begin{rmk}
\label{rem: degree tightness necessary}
The degree-tightness assumption in
\Cref{thm: resultant distribution convergence introduction} cannot be omitted. Indeed, let $S=\{1\}$, so that, according to our convention, $\mathcal{U}_1$ is the open unit disk, and consider the deterministic
sequence
$$
P^S_n(x):=x^n-p\in\mathscr{P}^{\{1\}}.
$$
For every fixed $Z\in\mathscr{P}^{\{1\}}$, we claim that
$$
\val\bigl(\Res(P^S_n,Z)\bigr)=\deg Z
$$
for all sufficiently large $n$. To see this, let $\alpha_1,\ldots,\alpha_d$ be the roots of $Z$, counted
with multiplicity. Since each $\alpha_j$ lies in the open unit disk, we have
$\val(\alpha_j)>0$. Hence, for all sufficiently large $n$,
$$
n\val(\alpha_j)>1
$$
for every $j\in\{1,\ldots,d\}$, and therefore
$\val(\alpha_j^n-p)=1$. It follows that
$$
\val\bigl(\Res(P^S_n,Z)\bigr)=\sum_{j=1}^d\val(\alpha_j^n-p)=d=\deg Z.
$$

Thus, for every fixed test polynomial $Z\in\mathscr{P}^{\{1\}}$, the
resultant valuation converges, while
$\deg P^S_n=n$ is not tight. This proves that the tightness assumption \eqref{item: degree tightness} in \Cref{thm: resultant distribution convergence introduction} is necessary. Similar examples can be constructed in any other lifted subspace by replacing $x$ with a fixed monic lift of the corresponding irreducible polynomial over $\F_p$.
\end{rmk}

We now turn to the convergence of the expected root statistics. To pass from weak convergence of the random polynomials to weak convergence of their root-counting statistics, we first need to identify the continuity points of the map $Z_U$. Although $Z_U$ may fail to be continuous at polynomials with repeated roots, the following lemma shows that it is continuous at every squarefree polynomial. This will allow us to apply the continuous mapping theorem under the squarefreeness assumption on $P^S_\infty$. The stronger degree condition in \Cref{thm: expected root statistics convergence} will then provide the uniform integrability needed to pass from convergence in distribution to convergence of expectations.

\begin{lemma}[Continuity of root statistics at squarefree polynomials]
\label{lem: continuity of root statistics}
Let $K_1,\ldots,K_m$ be finite extensions of $\Q_p$, and let
$$
U
\subseteq
\mc{O}_{K_1}^{\new}
\times\cdots\times
\mc{O}_{K_m}^{\new}
$$
be clopen. Then the map
$$
Z_U:\mathscr{P}^S\longrightarrow\Z_{\geq 0}
$$
is continuous at every squarefree polynomial
$P^S\in\mathscr{P}^S$.
\end{lemma}

\begin{proof}
Let $P^S\in\mathscr{P}^S$ be squarefree, and let
$Q_r^S\in\mathscr{P}^S$ be a sequence such that
$$
Q_r^S\longrightarrow P^S.
$$
Since $\mathscr{P}^S$ is equipped with the disjoint-union topology, we
have
$$
\deg Q_r^S=\deg P^S
$$
for all sufficiently large $r$.

Let $\mathcal R(P^S)\subseteq\overline{\Q}_p$
denote the set of roots of $P^S$. Since $P^S$ is squarefree, its roots
are pairwise distinct. By continuity of the roots of a monic polynomial
with respect to its coefficients, after labeling the roots of $Q_r^S$
appropriately, we may write
$$
P^S(x)=\prod_{a\in\mathcal R(P^S)}(x-a),
\qquad Q_r^S(x)=\prod_{a\in\mathcal R(P^S)}(x-a_r),
$$
where $a_r\longrightarrow a$ for every $a\in\mathcal R(P^S)$.

We first record that this correspondence preserves Galois orbits for all
sufficiently large $r$. Choose pairwise disjoint sufficiently small open
balls around the roots of $P^S$, with the balls chosen equivariantly
under the action of
$\Gal(\overline{\Q}_p/\Q_p)$. For all sufficiently large $r$, each such
ball contains exactly one root of $Q_r^S$. Hence, for every
$\sigma\in\Gal(\overline{\Q}_p/\Q_p)$, uniqueness of the root in the
corresponding ball gives
$$
(\sigma a)_r=\sigma(a_r).
$$
Thus roots lying in the same Galois orbit for $P^S$ correspond to roots
lying in the same Galois orbit for $Q_r^S$.

Conversely, let $\mathscr O_1$ and $\mathscr O_2$ be two distinct Galois
orbits of roots of $P^S$. The unions of the chosen balls around the roots
in $\mathscr O_1$ and $\mathscr O_2$ are disjoint and Galois invariant.
Therefore a root of $Q_r^S$ lying in the first union cannot be Galois
conjugate to a root lying in the second. It follows that, for all
sufficiently large $r$,
$$
a_r\text{ and }b_r\text{ are Galois conjugate over }\Q_p
\quad\Longleftrightarrow\quad
a\text{ and }b\text{ are Galois conjugate over }\Q_p.
$$

We next show that membership in each of the prescribed extensions is
stable. Fix $j\in\{1,\ldots,m\}$. Suppose first that
$a\in K_j$. By shrinking the ball around $a$ if necessary, Krasner's
lemma gives
$$
\Q_p(a)\subseteq\Q_p(a_r)
$$
for all sufficiently large $r$. Let $\mathscr O(a)$ denote the Galois
orbit of $a$. The union of the chosen balls around the elements of
$\mathscr O(a)$ contains exactly
$$
[\Q_p(a):\Q_p]
$$
roots of $Q_r^S$. Since the entire Galois orbit of $a_r$ lies in this
union, we have
$$
[\Q_p(a_r):\Q_p]
\leq
[\Q_p(a):\Q_p].
$$
Together with Krasner's lemma, this yields
$$
\Q_p(a_r)=\Q_p(a).
$$
In particular,
$$
a_r\in\Q_p(a)\subseteq K_j.
$$

Conversely, if $a_r\in K_j$ along an infinite subsequence, then
$a_r\to a$ and $K_j$ is closed in $\overline{\Q}_p$, so
$a\in K_j$. Consequently, for every $j$ and all sufficiently large $r$,
the correspondence
$$
a\longleftrightarrow a_r
$$
identifies the roots of $P^S$ lying in $K_j$ with the roots of $Q_r^S$
lying in $K_j$. Moreover, the equality of generated fields above implies
that
$$
a\in\mc{O}_{K_j}^{\new}
\quad\Longleftrightarrow\quad
a_r\in\mc{O}_{K_j}^{\new}.
$$

It remains to check membership in $U$. Consider a tuple
$$
(a_1,\ldots,a_m)
\in
\mc{O}_{K_1}^{\new}
\times\cdots\times
\mc{O}_{K_m}^{\new}
$$
formed from roots of $P^S$. Then
$$
(a_{1,r},\ldots,a_{m,r})
\longrightarrow
(a_1,\ldots,a_m).
$$
Since $U$ is clopen in
$$
\mc{O}_{K_1}^{\new}
\times\cdots\times
\mc{O}_{K_m}^{\new},
$$
we have
$$
(a_{1,r},\ldots,a_{m,r})\in U
\quad\Longleftrightarrow\quad
(a_1,\ldots,a_m)\in U
$$
for all sufficiently large $r$.

There are only finitely many tuples of roots of $P^S$ to consider, so
the same sufficiently large $r$ may be chosen simultaneously for all of
them. Combining this stability of membership in $U$ with the preservation
of Galois conjugacy proved above, we obtain
$$
Z_U(Q_r^S)=Z_U(P^S)
$$
for all sufficiently large $r$. Hence $Z_U$ is continuous at $P^S$.
\end{proof}

We now turn from weak convergence of the random polynomials themselves to
the corresponding root-counting statistics. The first step is to show that,
provided the limiting polynomial is almost surely squarefree, weak convergence
of the polynomials implies weak convergence of the random variables
$Z_U(P_n^S)$. Under the stronger moment-tightness assumption, this convergence
can then be upgraded to convergence of expectations.

\begin{proof}[Proof of
\Cref{thm: expected root statistics convergence}]
By \Cref{thm: resultant distribution convergence introduction}, we have
$P_n^S\overset{d}{\longrightarrow}P_\infty^S$
as random elements of $\mathscr{P}^S$. By
\Cref{lem: continuity of root statistics}, the map
$$
Z_U:\mathscr{P}^S\longrightarrow\Z_{\geq 0}
$$
is continuous at every squarefree polynomial in $\mathscr{P}^S$.
Since, by assumption,
$$
\mathbf{P}\left(
P_\infty^S\text{ is squarefree}
\right)=1,
$$
the map $Z_U$ is continuous at $P_\infty^S$ almost surely. The
continuous mapping theorem therefore gives
$$
Z_U(P_n^S)
\overset{d}{\longrightarrow}
Z_U(P_\infty^S).
$$
This proves the first assertion. Suppose now, in addition, that the stronger  moment-tightness condition holds. For every $P^S\in\mathscr{P}^S$, we have
$$
0\leq Z_U(P^S)\leq(\deg P^S)^m,
$$
since $Z_U(P^S)$ counts certain ordered $m$-tuples of roots of $P^S$.
Consequently, for every $D\geq 1$,
$$
Z_U(P_n^S)\mathbf{1}_{\{Z_U(P_n^S)>D^m\}}
\leq(\deg P_n^S)^m\mathbf{1}_{\{\deg P_n^S>D\}}.
$$
Indeed, if $\deg P_n^S\leq D$, then
$$
Z_U(P_n^S)\leq(\deg P_n^S)^m\leq D^m.
$$
Hence the moment-tightness assumption implies
$$
\lim_{D\to\infty}
\sup_{n\geq 1}
\mathbf{E}\left[
Z_U(P_n^S)
\mathbf{1}_{\{Z_U(P_n^S)>D^m\}}
\right]
=0.
$$
Thus the sequence $\left\{Z_U(P_n^S)
\right\}_{n\geq 1}$ is uniformly integrable.
Combining the weak convergence
$$
Z_U(P_n^S)
\overset{d}{\longrightarrow}
Z_U(P_\infty^S)
$$
with uniform integrability, we conclude that
$$
\lim_{n\to\infty}
\mathbf{E}\left[
Z_U(P_n^S)
\right]
=
\mathbf{E}\left[
Z_U(P_\infty^S)
\right].
$$
This proves the second assertion.
\end{proof}

\begin{rmk}
\label{rem: assumptions for expected root statistics necessary}
Neither the squarefreeness assumption on $P^S_\infty$ nor the stronger moment tightness assumption in
\Cref{thm: expected root statistics convergence} can be omitted.

We first show that the squarefreeness assumption is necessary. Let
$m=1$, $K_1=\Q_p$, $S=\{1\}$, and $U=p\Z_p$. Consider the deterministic
sequence
$$
P^S_n(x):=x^2-p^{2n+1}.
$$
The two roots of $P^S_n$ have valuation $n+\frac{1}{2}$, so they lie in the
open unit disk, but neither root belongs to $\Q_p$. Hence $Z_U(P^S_n)=0$ for every $n$. On the other hand,
$$
P^S_n\longrightarrow P^S_\infty:=x^2
$$
in $\mathscr{P}^{\{1\}}$, and this limit polynomial has a root in $p\Z_p$. Thus the conclusion of
\Cref{thm: expected root statistics convergence} fails, even though the
degrees are uniformly bounded. 

We next show that ordinary degree tightness is not sufficient in place of
the stronger assumption. Again take $m=1$, $K_1=\Q_p$, $S=\{1\}$, and
$U=p\Z_p$. Let
$$
P^S_\infty(x):=x-p
$$
and, for every $n\geq1$, define
$$
H_n(x):=\prod_{j=1}^n(x-pj).
$$
All the roots of $H_n$ are distinct and lie in $p\Z_p$. Let $P^S_n$ be the
random polynomial given by
$$
P^S_n=
\begin{cases}
H_n, & \text{with probability }n^{-1},\\
P^S_\infty, & \text{with probability }1-n^{-1}.
\end{cases}
$$
Then $P^S_n\overset{d}{\longrightarrow}P^S_\infty$, and consequently all the resultant valuations converge in distribution.
Moreover, the sequence $\{\deg P^S_n\}_{n\geq1}$ is tight. However, the stronger degree condition fails. Indeed,
$$
\sup_{n\geq1}
\mathbf{E}\left[
\deg P^S_n\mathbf{1}_{\{\deg P^S_n>D\}}
\right]
\geq1
$$
for every $D\geq1$, by choosing any integer $n>D$. Correspondingly,
$$
\mathbf{E}[Z_U(P^S_n)]=\left(1-\frac{1}{n}\right)Z_U(P^S_\infty)+\frac{1}{n}Z_U(H_n)
=2-\frac{1}{n},
$$
whereas $\mathbf{E}[Z_U(P^S_\infty)]=1$. Thus weak convergence together with ordinary degree tightness does not imply convergence of the expected root statistics. The stronger assumption in \Cref{thm: expected root statistics convergence} is needed to rule out
the contribution of rare polynomials having very large degree and many
roots in $U$.
\end{rmk}

%% file: The_random_matrix_case.tex
\section{The random matrix case}\label{sec: The random matrix case}

\begin{lemma}
\label{lem: resultant characteristic polynomial matrix evaluation}
Fix an integer $n\geq 1$ and a matrix $A_n\in\Mat_n(\Z_p)$. Then, for every polynomial $Z\in\Z_p[x]$, we have
$$
\Res(P_{A_n},Z)=\det\bigl(Z(A_n)\bigr).
$$
\end{lemma}

\begin{proof}
Let $\lambda_1,\ldots,\lambda_n\in\overline{\Q}_p$ be the eigenvalues of the fixed matrix $A$, counted with algebraic multiplicity. Since $P_{A_n}$ is monic, the product formula for the resultant gives
$$
\Res(P_{A_n},Z)=\prod_{i=1}^n Z(\lambda_i).
$$

On the other hand, over $\overline{\Q}_p$, the matrix $A$ is conjugate to an upper-triangular matrix whose diagonal entries are $\lambda_1,\ldots,\lambda_n$. Consequently, $Z(A_n)$ is conjugate to an upper-triangular matrix whose diagonal entries are
$Z(\lambda_1),\ldots,Z(\lambda_n)$. It follows that
$$
\det\bigl(Z(A_n)\bigr)=\prod_{i=1}^n Z(\lambda_i)=\Res(P_{A_n},Z),
$$
as desired.
\end{proof}

The identity in \Cref{lem: resultant characteristic polynomial matrix evaluation}
allows us to reinterpret the resultant as a determinant of a polynomial
evaluation of a random matrix. Consequently, the convergence of the resultant
distributions follows from the recent universality theorem of Cheong and Yu for
the cokernels of polynomial evaluations of random matrices. This is the key input that verifies the first hypothesis of
\Cref{thm: resultant distribution convergence introduction}.

\begin{prop}\label{prop: matrix resultant distribution convergence}
Let $Z\in\Z_p[x]$ be a fixed monic polynomial with degree $\ge 1$. Then, there exists a random variable $v=v(Z)\in\Z_{\ge 0}$ such that the following holds. For all $n\ge 1$, let $A_n\in\Mat_n(\Z_p)$ be random, with its entries independent and $\epsilon$-balanced. Then, we have a weak convergence
$$\val(\Res(P_{A_n},Z))\overset{d}{\longrightarrow}v,\qquad n\rightarrow\infty.$$
\end{prop}

\begin{proof}
By \Cref{lem: resultant characteristic polynomial matrix evaluation}, we have
$\Res(P_{A_n},Z)=\det\bigl(Z(A_n)\bigr)$. Since $Z(A_n)$ is a square matrix over $\Z_p$, its cokernel is finite if and only if
$\det(Z(A_n))\neq0$. On this event, the Smith normal form gives
$$
\val\det\bigl(Z(A_n)\bigr)=\log_p\left(\#\Cok\bigl(Z(A_n)\bigr)\right).
$$
If $\det(Z(A_n))=0$, then $\Cok(Z(A_n))$ is infinite and
$$
\val\det\bigl(Z(A_n)\bigr)=\infty.
$$
We now apply \cite[Theorem~1.3]{cheong2023distribution} to the polynomial $Z$.
That theorem asserts that $\Cok(Z(A_n))$, viewed as a random
$\Z_p[t]/(Z)$-module, converges weakly to a random finite
$\Z_p[t]/(Z)$-module $\mathcal{G}_Z$. In particular, since
$\mathcal{G}_Z$ is almost surely finite,
$$
\mathbf{P}\left(
\Cok\bigl(Z(A_n)\bigr)\text{ is infinite}
\right)
\longrightarrow 0.
$$
Consider the space consisting of the isomorphism classes of finite
$\Z_p[t]/(Z)$-modules together with one additional state, denoted by
$\infty$, representing an infinite cokernel. Define
$$
\Phi(G):=
\begin{cases}
\log_p\#G,&G\text{ is finite},\\
\infty,&G=\infty.
\end{cases}
$$
This is a continuous map into $\Z_{\geq0}\cup{\infty}$, where the latter is equipped with the one-point compactification topology. Therefore, the continuous mapping theorem gives
$$
\Phi\left(\Cok\bigl(Z(A_n)\bigr)\right)
\overset{d}{\longrightarrow}
\log_p\#\mathcal{G}_Z.
$$
By the Smith normal form identity above,
$$
\Phi\left(\Cok\bigl(Z(A_n)\bigr)\right)
=\val(\det\bigl(Z(A_n)\bigr))=
\val\bigl(\Res(P_{A_n},Z)\bigr).
$$
Hence, defining $v(Z):=\log_p\#\mathcal{G}_Z$,
we obtain
$$
\val\bigl(\Res(P_{A_n},Z)\bigr)
\overset{d}{\longrightarrow}
v(Z),
$$
as desired.
\end{proof}

It remains to verify the moment tightness condition required by
\Cref{thm: expected root statistics convergence}, as given by the following
proposition. Since its proof is largely independent of the rest of the argument and would interrupt the flow of the present section, we defer it to \Cref{sec: Moment Tightness}.

\begin{prop}\label{prop: degree moment tightness}
Let $S\subseteq \Z_{>0}$ be finite. For every $n\ge 1$, let
$A_n\in\Mat_n(\Z_p)$ be a random matrix whose entries are independent
and $\epsilon$-balanced. Then, for every integer $m\ge 1$,
$$
\lim_{D\to\infty}\sup_{n\ge 1}
\E\left[\left(\deg P_{A_n}^{S}\right)^m\mathbf 1_{\{\deg P_{A_n}^{S}>D\}}\right]=0.
$$
\end{prop}

To apply the convergence results of the resultant distribution method to
root statistics, we will also need to know that the limiting polynomial is
almost surely squarefree. For the additive Haar ensemble, the existence of
the limiting distribution already follows from the preceding resultant
convergence and degree-tightness estimates. The following lemma records
this fact together with the additional squarefreeness property. The latter
is proved using the same finite-dimensional reduction that underlies
\cite[Theorem~4.1]{ellenberg2011modeling}.

\begin{lemma}
\label{lem: haar matrix limit squarefree}
Fix a finite set $S\subseteq\Z_{>0}$. For every $n\geq 1$, let
$A_n^{\Haar}\in\Mat_n(\Z_p)$ be additive Haar distributed. Then there
exists a random polynomial
$$
Q_{\infty}^{\Haar,S}\in\mathcal P^S
$$
such that
$$
P_{A_n^{\Haar}}^S
\overset{d}{\longrightarrow}
Q_{\infty}^{\Haar,S},
\qquad n\to\infty.
$$
Moreover, the limiting polynomial is almost surely squarefree:
$$
\mathbf{P}\left(
Q_{\infty}^{\Haar,S}\text{ is squarefree}
\right)=1.
$$
\end{lemma}

\begin{proof}
We first prove the asserted weak convergence. Fix $Z\in\mathcal P^S$. By the lifted-subspace factorization, we may write
$$
P_{A_n^{\Haar}}=P_{A_n^{\Haar}}^S
P_{A_n^{\Haar}}^{\Z_{>0}\setminus S}.
$$
The reductions modulo $p$ of $Z$ and
$P_{A_n^{\Haar}}^{\Z_{>0}\setminus S}$ are supported on disjoint collections of irreducible polynomials. Hence, by \Cref{prop: properties of resultant} and
\Cref{prop: relatively prime resultant},
$$
\val\left(\Res(P_{A_n^{\Haar}}^S,Z)
\right)=\val\left(\Res(P_{A_n^{\Haar}},Z)
\right).
$$
By \Cref{prop: matrix resultant distribution convergence}, the random variables
on the right-hand side converge in distribution as $n\to\infty$.
Moreover, \Cref{prop: degree moment tightness}, applied with $m=1$,
implies that
$$
\left\{\deg P_{A_n^{\Haar}}^S\right\}_{n\geq 1}
$$
is tight. Therefore
\Cref{thm: resultant distribution convergence introduction} applies and yields
a random polynomial
$$
Q_{\infty}^{\Haar,S}\in\mathcal P^S
$$
such that $P_{A_n^{\Haar}}^S\overset{d}{\longrightarrow}Q_{\infty}^{\Haar,S}$.

It remains to prove that $Q_{\infty}^{\Haar,S}$ is almost surely
squarefree. We first reduce to the case in which $S$ consists of a single element. For every $i\in S$, consider the factor map
$$
\pi_i:\mathcal P^S\longrightarrow\mathcal P^{\{i\}},
\qquad
\pi_i(P):=P^{\{i\}}.
$$
This map is continuous. Indeed, the reduction modulo $p$ is locally
constant on each fixed-degree component, and the factors associated with
distinct lifted subspaces have relatively prime reductions modulo $p$.
The uniqueness in the strong form of Hensel's lemma therefore shows that
the corresponding lifted factors depend continuously on the coefficients.

It follows from the continuous mapping theorem that
$$
P_{A_n^{\Haar}}^{\{i\}}
=
\pi_i\left(P_{A_n^{\Haar}}^S\right)
\overset{d}{\longrightarrow}
\pi_i\left(Q_{\infty}^{\Haar,S}\right)
$$
for every $i\in S$. Thus it suffices to prove that, for every fixed
$i\geq 1$, every limiting polynomial of
$P_{A_n^{\Haar}}^{\{i\}}$ is almost surely squarefree. Indeed, since
$S$ is finite, this would imply that all the factors
$\pi_i(Q_{\infty}^{\Haar,S})$, $i\in S$, are simultaneously squarefree
with probability one. Their reductions modulo $p$ are supported on the
distinct irreducible polynomials $F_i$, so these factors are pairwise
relatively prime. Their product $Q_{\infty}^{\Haar,S}$ is therefore
squarefree.

Fix now $i\in S$, and denote $d_i:=\deg F_i$.
We have $P_{A_n^{\Haar}}^{\{i\}}\overset{d}{\longrightarrow}Q_\infty^{\Haar, \{i\}}$.
We analyze this convergence separately on each fixed-degree component $r\geq 0$. Since the reduction modulo $p$ of an
$\{i\}$-distinguished polynomial is a power of $F_i$, the event
$$
\deg P_{A_n^{\Haar}}^{\{i\}}=r
$$
is empty unless $d_i\mid r$. We may therefore suppose that $r=kd_i$ for some $k\geq 0$. Define
$$
\mathcal M_{r,i}:=\left\{
B\in\Mat_r(\Z_p):
\det(xI_r-B)\bmod p=F_i(x)^k
\right\}.
$$
This is a nonempty clopen subset of $\Mat_r(\Z_p)$. Let $\nu_{r,i}$
denote the pushforward to $\mathcal P_r^{\{i\}}$ of normalized additive
Haar measure on $\mathcal M_{r,i}$ under the characteristic-polynomial
map
$$
B\longmapsto P_B(x):=\det(xI_r-B).
$$

We claim that, for every $n\geq r$,
\begin{equation}
\label{eq: conditional singleton primary law}
\mathcal L\left(
P_{A_n^{\Haar}}^{\{i\}}
\,\middle|\,
\deg P_{A_n^{\Haar}}^{\{i\}}=r
\right)=\nu_{r,i}.
\end{equation}
To see this, fix $m\geq 1$ and put $R:=\Z/p^m\Z$. Let $A\in\Mat_n(R)$ be such that its $F_i$-primary component has rank
$r$. The Fitting decomposition associated with the $F_i$-primary part
gives a unique decomposition
$$
R^n=X\oplus Y
$$
into $A$-invariant free $R$-submodules such that $\rank_R X=r$, the characteristic polynomial of $A|_X$ reduces modulo $p$ to $F_i^k$, and the characteristic polynomial of $A|_Y$ is relatively prime to $F_i$
modulo $p$. Under this decomposition, the $F_i$-primary factor of the
characteristic polynomial of $A$ is precisely
$$
\det(xI_r-A|_X).
$$

The number of ordered decompositions
$$
R^n=X\oplus Y,
\qquad
\rank_R X=r,
$$
is
$$
\frac{|\GL_n(R)|}
{|\GL_r(R)|\,|\GL_{n-r}(R)|}.
$$
Once such a decomposition is fixed, the restrictions $A|_X$ and $A|_Y$
may be chosen independently, subject only to the two primary conditions
above.

Let $\mathcal U\subseteq\mathcal P_r^{\{i\}}$ be a clopen set whose
membership is determined modulo $p^m$. In the conditional probability
$$
\mathbf{P}\left(
P_{A_n^{\Haar}}^{\{i\}}\in\mathcal U
\,\middle|\,
\deg P_{A_n^{\Haar}}^{\{i\}}=r
\right),
$$
both the number of choices of the decomposition $X\oplus Y$ and the
number of admissible choices for the $Y$-block cancel between numerator
and denominator. We therefore obtain
$$
\mathbf{P}\left(
P_{A_n^{\Haar}}^{\{i\}}\in\mathcal U
\,\middle|\,
\deg P_{A_n^{\Haar}}^{\{i\}}=r
\right)=
\frac{
\#\left\{
B\in\Mat_r(R):
\det(xI_r-B)\bmod p=F_i^k,\
P_B\in\mathcal U
\right\}
}{
\#\left\{
B\in\Mat_r(R):
\det(xI_r-B)\bmod p=F_i^k
\right\}
}.
$$
Passing to the inverse limit over $m$ proves
\eqref{eq: conditional singleton primary law}.

We next show that $\nu_{r,i}$ assigns measure zero to the nonsquarefree
locus. This is immediate for $r=0,1$, so suppose that $r\geq 2$.
Consider
$$
\Delta_r(B)
:=
\Disc\left(\det(xI_r-B)\right),
\qquad
B\in\Mat_r(\Z_p).
$$
The function $\Delta_r$ is a polynomial in the $r^2$ entries of $B$, and
it is not identically zero. Indeed, if
$$
B=\diag(b_1,\ldots,b_r)
$$
for pairwise distinct $b_1,\ldots,b_r\in\Z_p$, then
$$
\det(xI_r-B)
=
\prod_{j=1}^r(x-b_j)
$$
is squarefree, and hence $\Delta_r(B)\neq 0$.

The zero set of a nonzero polynomial on $\Z_p^{r^2}$ has additive Haar
measure zero, thus the set
$$
\left\{B\in\mathcal M_{r,i}:
P_B\text{ is not squarefree}
\right\}
$$
has Haar measure zero. Since $\nu_{r,i}$ is the pushforward of normalized
Haar measure on the clopen set $\mathcal M_{r,i}$, we conclude that
\begin{equation}
\label{eq: fixed degree squarefree}
\nu_{r,i}
\left(
\left\{
P\in\mathcal P_r^{\{i\}}:
P\text{ is not squarefree}
\right\}
\right)
=0.
\end{equation}

It remains to pass this fixed-degree statement to the weak limit. Let
$\mu_n$ denote the law of $P_{A_n^{\Haar}}^{\{i\}}$, and let $\mu$
denote the law of $Q_\infty^{\Haar, \{i\}}$. By
\eqref{eq: conditional singleton primary law},
$$
\left.
\mu_n
\right|_{\mathcal P_r^{\{i\}}}
=
\mathbf{P}\left(
\deg P_{A_n^{\Haar}}^{\{i\}}=r
\right)\nu_{r,i}.
$$
Since $\mathcal P_r^{\{i\}}$ is open and closed in
$\mathcal P^{\{i\}}$, weak convergence implies
$$
\mathbf{P}\left(
\deg P_{A_n^{\Haar}}^{\{i\}}=r
\right)
\longrightarrow
\mathbf{P}\left(
Q_\infty^{\Haar, \{i\}}\in\mathcal P_r^{\{i\}}
\right).
$$
More generally, testing against bounded continuous functions supported on
$\mathcal P_r^{\{i\}}$ shows that
$$
\left.
\mu
\right|_{\mathcal P_r^{\{i\}}}
=
\mathbf{P}\left(
Q_\infty^{\Haar, \{i\}}\in\mathcal P_r^{\{i\}}
\right)\nu_{r,i}.
$$
Together with \eqref{eq: fixed degree squarefree}, this gives
$$
\mathbf{P}\left(
Q_\infty^{\Haar, \{i\}}\in\mathcal P_r^{\{i\}}
\text{ and }
Q_\infty^{\Haar, \{i\}}\text{ is not squarefree}
\right)
=0.
$$
Summing over $r\geq 0$, we obtain
$$
\mathbf{P}\left(
Q_\infty^{\Haar, \{i\}}\text{ is squarefree}
\right)=1.
$$
Applying this conclusion to every $i\in S$ and using the reduction at the
beginning of the proof yields
$$
\mathbf{P}\left(
Q_{\infty}^{\Haar,S}\text{ is squarefree}
\right)=1.
$$
This completes the proof.
\end{proof}

\begin{proof}[Proof of \Cref{thm: random matrix eigenvalue universality}]
Choose a finite set $S\subseteq\Z_{>0}$ such that
$$
\mc{O}_{K_j}^{\new}
\subseteq
\bigcup_{i\in S}\mc{U}_i,
\qquad
1\leq j\leq m.
$$
Such a finite set exists because each $K_j/\Q_p$ is finite, and hence
$\mc{O}_{K_j}^{\new}$ meets only finitely many lifted subspaces. By the choice of $S$, every root counted by $Z_U(P_{A_n})$
belongs to $P_{A_n}^S$. Therefore
\begin{equation}
\label{eq: root statistic distinguished factor matrix}
Z_U(P_{A_n})
=
Z_U(P_{A_n}^S).
\end{equation}

We first establish the weak convergence of $P_{A_n}^S$. Fix
$Z\in\mathscr{P}^S$. By the lifted-subspace factorization,
$$
P_{A_n}
=
P_{A_n}^S
P_{A_n}^{\Z_{>0}\setminus S}.
$$
The reductions modulo $p$ of $Z$ and
$P_{A_n}^{\Z_{>0}\setminus S}$ are supported on disjoint collections
of irreducible polynomials. Hence, by
\Cref{prop: properties of resultant,prop: relatively prime resultant},
$$
\val\left(
\Res(P_{A_n}^S,Z)
\right)
=
\val\left(
\Res(P_{A_n},Z)
\right).
$$
By \Cref{prop: matrix resultant distribution convergence}, there exists a random variable $v(Z):=\log_p\#G_Z$ such that
\begin{equation}
\label{eq: matrix distinguished resultant limit}
\val\left(
\Res(P_{A_n}^S,Z)
\right)
\overset{d}{\longrightarrow}
v(Z).
\end{equation}
In particular, the limiting distribution in
\eqref{eq: matrix distinguished resultant limit} depends only on $p$
and $Z$, and not on the distributions of the entries of $A_n$.

Next, by \Cref{prop: degree moment tightness}, for every integer
$r\geq 1$,
\begin{equation}
\label{eq: matrix moment tightness in universality proof}
\lim_{D\to\infty}\sup_{n\geq 1}
\E\left[\left(\deg P_{A_n}^S\right)^r
\mathbf{1}_{\{\deg P_{A_n}^S>D\}}
\right]=0.
\end{equation}
Taking $r=1$ shows in particular that
$\left\{\deg P_{A_n}^S\right\}_{n\geq 1}$
is tight. We may therefore apply
\Cref{thm: resultant distribution convergence introduction}. There exists
a random polynomial $Q_\infty^S\in\mathscr{P}^S$ such that
\begin{equation}
\label{eq: general matrix distinguished weak limit}
P_{A_n}^S\overset{d}{\longrightarrow}
Q_\infty^S.
\end{equation}
Moreover, for every fixed $Z\in\mathscr{P}^S$,
$$
\val\left(\Res(Q_\infty^S,Z)
\right)\overset{d}{=}v(Z).
$$

We now identify this limiting law with the additive Haar limit. Let
$$
A_n^{\Haar}\in\Mat_n(\Z_p)
$$
be additive Haar distributed. By
\Cref{lem: haar matrix limit squarefree}, there exists $Q_\infty^{\Haar,S}\in\mathscr{P}^S$ such that
$$
P_{A_n^{\Haar}}^S
\overset{d}{\longrightarrow}
Q_\infty^{\Haar,S},
$$
and
\begin{equation}
\label{eq: Haar limit squarefree in universality proof}
\mathbf{P}\left(Q_\infty^{\Haar,S}\text{ is squarefree}\right)=1.
\end{equation}
Applying \Cref{prop: matrix resultant distribution convergence} to the
Haar ensemble gives, for every $Z\in\mathscr{P}^S$,
$$
\val\left(\Res(Q_\infty^{\Haar,S},Z)
\right)\overset{d}{=}v(Z).
$$
Consequently,
$$
\val\left(\Res(Q_\infty^S,Z)
\right)\overset{d}{=}\val\left(
\Res(Q_\infty^{\Haar,S},Z)\right)
$$
for every $Z\in\mathscr{P}^S$. By
\Cref{thm: resultant distribution determination introduction},
\begin{equation}
\label{eq: universal matrix polynomial limit}
Q_\infty^S\overset{d}{=}Q_\infty^{\Haar,S}.
\end{equation}
In particular, \eqref{eq: Haar limit squarefree in universality proof}
and \eqref{eq: universal matrix polynomial limit} imply that $\mathbf{P}\left(Q_\infty^S\text{ is squarefree}\right)=1$.

We may now pass to the expected eigenvalue statistics. Applying the
convergence-of-expected-root-statistics part of the resultant distribution
method to \eqref{eq: general matrix distinguished weak limit}, using
\eqref{eq: matrix moment tightness in universality proof} with
$r=m$ and the almost-sure squarefreeness of $Q_\infty^S$, gives
\begin{equation}
\label{eq: general matrix expectation to limit polynomial}
\lim_{n\to\infty}
\E\left[Z_U(P_{A_n}^S)\right]=\E\left[Z_U(Q_\infty^S)\right].
\end{equation}
By \eqref{eq: universal matrix polynomial limit},
$$
\E\left[Z_U(Q_\infty^S)
\right]=\E\left[Z_U(Q_\infty^{\Haar,S})
\right].
$$

On the other hand, the same convergence-of-expectations result applied
to the Haar ensemble yields
$$
\E\left[Z_U(Q_\infty^{\Haar,S})
\right]=\lim_{n\to\infty}
\E\left[Z_U(P_{A_n^{\Haar}}^S)\right].
$$
Again by the choice of $S$, $
Z_U(P_{A_n^{\Haar}}^S)=Z_U(P_{A_n^{\Haar}})$. The limiting correlation-function theorem for additive Haar random
matrices from \Cref{thm: Haar random matrix correlations} therefore gives
$$
\E\left[Z_U(Q_\infty^{\Haar,S})
\right]=\lim_{n\to\infty}\E\left[Z_U(P_{A_n^{\Haar}})
\right]=\int_U\rho_{K_1,\ldots,K_m}^{(\infty)}(x_1,\ldots,x_m)\,dx_1\cdots dx_m.
$$
Combining this identity with
\eqref{eq: root statistic distinguished factor matrix} and
\eqref{eq: general matrix expectation to limit polynomial}, we conclude
that
$$
\lim_{n\to\infty}\E\left[Z_U(P_{A_n})
\right]=\int_U\rho_{K_1,\ldots,K_m}^{(\infty)}(x_1,\ldots,x_m)\,dx_1\cdots dx_m.
$$
This proves \Cref{thm: random matrix eigenvalue universality}.
\end{proof}

\begin{rmk}
\label{rem: clopen assumption is necessary}
The clopen assumption on $U$ in \Cref{thm: random matrix eigenvalue universality}
cannot in general be replaced by the assumption that $U$ is merely open or
merely closed.

For every $n\geq 1$, let $A_n\in\Mat_n(\Z_p)$ have independent Bernoulli
entries, each taking the values $0$ and $1$ with equal probability. In
particular, the entries are $1/2$-balanced for every prime $p$. Let
$$
V:=\left\{
x\in\Z_p:
x\text{ is a }\Q_p\text{-eigenvalue of some }
A\in\Mat_n(\{0,1\})
\text{ for some }n\geq 1
\right\}.
$$
For each $n$, there are only finitely many matrices in
$\Mat_n(\{0,1\})$, and each of their characteristic polynomials has only
finitely many roots. Hence $V$ is countable and therefore has Haar measure
zero.

Fix $\varepsilon>0$. By outer regularity of Haar measure, there exists an
open set $V\subset U\subset\Z_p$ such that
$$
\mu(U)=\mu(U\setminus V)<\varepsilon.
$$
Every $\Q_p$-eigenvalue of $A_n$ belongs to $\Z_p$, since its characteristic
polynomial $P_{A_n}$ is monic with coefficients in $\Z_p$. By the definition
of $V$, every such eigenvalue also belongs to $V$. Consequently,
$$
Z_U(P_{A_n})=Z_{\Z_p}(P_{A_n})
$$
almost surely. Applying \Cref{thm: random matrix eigenvalue universality} to the
clopen set $\Z_p$, together with
$\rho_{\Q_p}^{(\infty)}(x)=1$ for $x\in\Z_p$, gives
$$
\lim_{n\to\infty}\E\!\left[Z_U(P_{A_n})\right]=\lim_{n\to\infty}\E\!\left[Z_{\Z_p}(P_{A_n})\right]=1.
$$
On the other hand, if the conclusion of
\Cref{thm: random matrix eigenvalue universality} remained valid for arbitrary open
sets, then it would give
$$
\lim_{n\to\infty}\E\!\left[Z_U(P_{A_n})\right]=\int_U \rho_{\Q_p}^{(\infty)}(x)\,dx=\mu(U)<\varepsilon.
$$
Choosing $0<\varepsilon<1$ yields a contradiction. Thus openness alone is
not sufficient.

The same construction also shows that closedness alone is not sufficient.
Let $W:=\Z_p\setminus U$. Then $W$ is closed. Since $V\subset U$, no
$\Q_p$-eigenvalue of $A_n$ lies in $W$, and therefore
$Z_W(P_{A_n})=0$ almost surely for every $n$. Hence
$$
\lim_{n\to\infty}\E\!\left[Z_W(P_{A_n})\right]=0.
$$
However,
$$
\int_W \rho_{\Q_p}^{(\infty)}(x)\,dx=\mu(W)
=1-\mu(U)>1-\varepsilon,
$$
which is strictly positive for $0<\varepsilon<1$. Thus the conclusion of
\Cref{thm: random matrix eigenvalue universality} may fail when $U$ is open but not
closed, and also when $U$ is closed but not open.
\end{rmk}

We conclude the random matrix application by deriving the two explicit
consequences stated in the introduction. Once \Cref{thm: random matrix eigenvalue universality} is available, no further
universality argument is needed: it remains only to specialize the limiting
correlation functions and invoke the explicit computations for the additive
Haar ensemble obtained in \cite{shen2026eigenvalues}. For the second moment over $\Z_p$, we pass from the two-point correlation function to the ordinary
second moment through the corresponding factorial moment. 

\begin{proof}[Proof of \Cref{thm: Zp first second moments universality}]
Since $\mathcal{O}_{\Q_p}^{\mathrm{new}}=\Z_p$, applying
\Cref{thm: random matrix eigenvalue universality} with $m=1$, $K_1=\Q_p$, and
$U=\Z_p$ gives
$$
\lim_{n\to\infty}
\E\!\left[Z_{\Z_p}(P_{A_n})\right]
=\int_{\Z_p}\rho_{\Q_p}^{(\infty)}(x)\,dx.
$$
By \cite[Theorem 9.1]{shen2026eigenvalues},
$\rho_{\Q_p}^{(\infty)}(x)=1$ for every $x\in\Z_p$. Since the Haar measure
on $\Z_p$ is normalized to have total mass one, we obtain
$$
\lim_{n\to\infty}\E\!\left[Z_{\Z_p}(P_{A_n})\right]=1.
$$

For the second moment, by the definition of the joint root-counting statistic,
$$
Z_{\Z_p^2}(P_{A_n})
=
Z_{\Z_p}(P_{A_n})
\left(
Z_{\Z_p}(P_{A_n})-1
\right),
$$
since $Z_{\Z_p^2}(P_{A_n})$ counts ordered pairs of distinct
$\Q_p$-eigenvalues. Hence
$$
Z_{\Z_p}(P_{A_n})^2=Z_{\Z_p^2}(P_{A_n})
+Z_{\Z_p}(P_{A_n}).
$$
Applying \Cref{thm: random matrix eigenvalue universality} with
$m=2$, $K_1=K_2=\Q_p$, and $U=\Z_p^2$, and combining this with the
first-moment convergence above, we obtain
$$
\lim_{n\to\infty}
\E\!\left[Z_{\Z_p}(P_{A_n})^2
\right]=\int_{\Z_p^2}
\rho_{\Q_p,\Q_p}^{(\infty)}(x,y)\,dx\,dy
+\int_{\Z_p}
\rho_{\Q_p}^{(\infty)}(x)\,dx.
$$
The right-hand side is the limiting second moment for the additive Haar
ensemble. By \cite[Corollary 9.6]{shen2026eigenvalues}, it is equal to
the series displayed in \Cref{thm: Zp first second moments universality}. This proves the result.
\end{proof}

\begin{proof}[Proof of \Cref{thm: quadratic extension universality}]
Let $K/\Q_p$ be quadratic. We apply
\Cref{thm: random matrix eigenvalue universality} with $m=1$, $K_1=K$, and
$U=\mathcal{O}_K^{\mathrm{new}}$. Since
$\mathcal{O}_K^{\mathrm{new}}$ is the entire ambient space under
consideration, it is clopen in itself. Hence
$$
\lim_{n\to\infty}
\E\!\left[Z_{\mathcal{O}_K^{\mathrm{new}}}(P_{A_n})\right]=\int_{\mathcal{O}_K^{\mathrm{new}}}\rho_K^{(\infty)}(x)\,dx.
$$
The integral on the right is the limiting expected number of eigenvalues
generating $K$ for the additive Haar matrix ensemble. It was evaluated
explicitly in \cite[Theorem 1.10]{shen2026eigenvalues}. In the unramified
case it gives the first series displayed in \Cref{thm: quadratic extension universality}, while in the
ramified case it gives the second series, including the factor
$\|\operatorname{Disc}_{K/\Q_p}\|$. These are exactly the two claimed
formulas.
\end{proof}

%% file: Growing_Surjection_Moments_and_Moment_Tightness.tex
\section{Moment Tightness via Growing Surjection Moments}
\label{sec: Moment Tightness}

The purpose of this section is to prove
\Cref{prop: degree moment tightness}. The main input is a growing-target
version of the usual surjection-moment estimate for random finite-field
matrix modules. In contrast with the fixed-target moment method, we allow
the dimension of the target module to grow logarithmically with the matrix
size. This stronger estimate gives quantitative tail bounds for the primary
partitions of the reduction modulo $p$, from which the desired moment
tightness follows.

We begin by recalling some notation for partitions.

\begin{defi}
We denote by
$$
\Y
=
\left\{
\lambda=(\lambda_1,\lambda_2,\ldots):
\lambda_1\geq\lambda_2\geq\cdots\geq 0,\ 
\lambda_i\in\Z,\ 
\lambda_i=0\text{ for all but finitely many }i
\right\}
$$
the set of integer partitions. The positive integers $\lambda_i$ are called
the \emph{parts} of $\lambda$. For $\lambda\in\Y$, we write
$$
|\lambda|
:=
\sum_{j\geq 1}\lambda_j,
\qquad
\ell(\lambda)
:=
\#\{j:\lambda_j>0\}.
$$
We refer to $\lambda_1$ as the \emph{largest part} of $\lambda$, with the
convention that $\lambda_1=0$ for the zero partition.
\end{defi}

For every $n\geq 1$, write
$$
B_n:=A_n\bmod p\in\Mat_n(\F_p),
\qquad
M_n:=\Cok_{\F_p[t]}(tI_n-B_n).
$$
Recall that $F_1,F_2,\ldots$ denotes the fixed enumeration of monic
irreducible polynomials in $\F_p[t]$, and write $d_i:=\deg F_i$. For each $i\geq 1$, let
$$
\lambda^{(i)}(B_n)=\left(\lambda^{(i)}_1(B_n),
\lambda^{(i)}_2(B_n),\ldots
\right)\in\Y
$$
be the partition associated with $F_i$ in the rational canonical form of $B_n$. Equivalently, the $F_i$-primary component of the $\F_p[t]$-module $M_n$ is isomorphic to
$$
\bigoplus_{j\geq 1}
\F_p[t]/\left(F_i^{\lambda^{(i)}_j(B_n)}\right).
$$
By the definition of the lifted-subspace factorization, for every finite
$S\subseteq\Z_{>0}$,
\begin{equation}
\label{eq: degree from primary partitions}
\deg P_{A_n}^{S}=\sum_{i\in S}
d_i\left|\lambda^{(i)}(B_n)\right|.
\end{equation}

We shall prove the following quantitative estimate, which is stronger than
what is needed for \Cref{prop: degree moment tightness}.

\begin{prop}[Quantitative degree tail]
\label{prop: quantitative degree tail}
Let $S\subseteq\Z_{>0}$ be finite. For every $n\geq 1$, let
$A_n\in\Mat_n(\Z_p)$ be a random matrix whose entries are independent and
$\epsilon$-balanced. Then, for every $M>0$, there exist constants
$C_{S,M},c_{S,M}>0$ such that, for every $D\geq 2$,
\begin{equation}
\label{eq: quantitative degree tail}
\sup_{n\geq 1}\mathbf{P}\left(
\deg P_{A_n}^{S}>D\right)\leq
C_{S,M}D^{-M}+C_{S,M}\exp\left(
-c_{S,M}\frac{D}{\log(D+2)}
\right).
\end{equation}
\end{prop}

The remainder of the section is devoted to proving \Cref{prop: quantitative degree tail}. We first establish a uniform
surjection-moment estimate for targets whose dimension grows logarithmically
with $n$.

\subsection{Growing-target surjection moments}

We begin with the finite-field estimate that drives the argument.

\begin{thm}[Growing-target surjection moments]
\label{thm: growing target surjection moments}
Fix $L>0$. There exist constants
$c=c(p,\epsilon,L)>0$, $C=C(p,\epsilon,L)>0$, and
$n_0=n_0(p,\epsilon,L)$ such that the following holds.
For every $n\geq n_0$ and every finite $\F_p[t]$-module $G$ satisfying
$$
r:=\dim_{\F_p}G\leq L\log(n+2),
$$
we have
\begin{equation}
\label{eq: growing target surjection moment}
\left|\E\left[\#\Sur_{\F_p[t]}(M_n,G)
\right]-1\right|\leq C\exp\left(-c\frac{n}{\log(n+2)}\right).
\end{equation}
In particular, for every fixed $L>0$,
\begin{equation}
\label{eq: uniformly bounded growing surjection moments}
\sup_{\substack{n\geq 1\\
\dim_{\F_p}G\leq L\log(n+2)}}
\E\left[\#\Sur_{\F_p[t]}(M_n,G)\right]
<\infty.
\end{equation}
\end{thm}

The proof of \Cref{thm: growing target surjection moments} is based on a
Fourier estimate for linear maps whose columns are sufficiently
unstructured, together with a peeling argument for the remaining maps.

\subsection{Fourier and combinatorial estimates}

Let
$$
e_p(u):=\exp\left(\frac{2\pi\sqrt{-1}}{p}u\right),\qquad u\in\F_p.
$$

We first record the following explicit Fourier estimate.

\begin{lemma}[Fourier estimate]
\label{lem: uniform Fourier gap}
Let $\xi$ be an $\epsilon$-balanced $\F_p$-valued random variable.
Then, for every $a\in\F_p^\times$,
\begin{equation}
\label{eq: explicit Fourier gap}
\left|\E e_p(a\xi)\right|\leq
\exp\left(-\frac{\epsilon}{p^2}\right).
\end{equation}
\end{lemma}

\begin{proof}
Since $a\neq 0$, the complex number $e_p(a)$ is a nontrivial $p$-th root of unity. The claim follows immediately from
\cite[Lemma~4.2]{wood2017distribution}, applied with $b=p$ and
$\alpha=\epsilon$.
\end{proof}

Let $V$ be a finite-dimensional $\F_p$-vector space, let $I$ be a finite
index set, and let
$$
C:\F_p^I\longrightarrow V
$$
be linear. Write $C_i:=C(e_i)$ for its columns.

\begin{defi}
\label{def: code distance}
Let
$$
H:=\Span\{C_i:i\in I\}.
$$
We say that $C$ is a \emph{code of distance at least $h$ onto $H$} if,
for every nonzero $\varphi\in H^*$,
$$
\#\{i\in I:\varphi(C_i)\neq 0\}\geq h.
$$
\end{defi}

\begin{lemma}[Fourier equidistribution for a code]
\label{lem: Fourier equidistribution for a code}
Suppose that $C:\F_p^I\to H$ is a code of distance at least $h$ onto an
$s$-dimensional vector space $H$. Let
$X=(X_i)_{i\in I}$ have independent $\epsilon$-balanced coordinates.
Then, for every $y\in H$,
\begin{equation}
\label{eq: code Fourier bound}
\left|
\mathbf{P}(CX=y)-p^{-s}
\right|
\leq
\exp\left(-\frac{\epsilon h}{p^2}\right).
\end{equation}
\end{lemma}

\begin{proof}
Fourier inversion on the additive group of $H$ gives
$$
\mathbf{P}(CX=y)
=
p^{-s}
\sum_{\varphi\in H^*}
e_p(-\varphi(y))
\prod_{i\in I}
\E e_p\left(\varphi(C_i)X_i\right).
$$
The contribution of $\varphi=0$ is $p^{-s}$. If $\varphi\neq 0$, then
$\varphi(C_i)\neq 0$ for at least $h$ indices $i$. By
\eqref{eq: explicit Fourier gap}, every corresponding Fourier factor has
modulus at most $\exp(-\epsilon/p^2)$. Therefore the total contribution
of all nonzero $\varphi$ has modulus at most
$$
p^{-s}(p^s-1)
\exp\left(-\frac{\epsilon h}{p^2}\right)
\leq
\exp\left(-\frac{\epsilon h}{p^2}\right).
$$
\end{proof}

We next show that every linear map becomes a code after removing a
controlled number of columns.

\begin{lemma}[Peeling lemma]
\label{lem: peeling noncode map}
Let $C:\F_p^n\to V$ be surjective, where $\dim_{\F_p}V=r$, and let
$h\geq 1$. Then there exist a set $S_0\subseteq[n]$ and a subspace
$H\leq V$ such that, writing
$$
d:=\operatorname{codim}_V H,
$$
we have
\begin{equation}
\label{eq: peeling size}
|S_0|<dh,
\end{equation}
$$
H=\Span\{C_i:i\notin S_0\},
$$
and the restriction
$$
C|_{\F_p^{S_0^c}}:\F_p^{S_0^c}\longrightarrow H
$$
is a code of distance at least $h$.
\end{lemma}

\begin{proof}
Start with $S_0=\varnothing$ and current span $H_0=V$. If the current
restriction is not a code of distance $h$, choose a nonzero functional
$\varphi\in H_0^*$ whose support on the current columns has cardinality
less than $h$, and remove all columns on which $\varphi$ is nonzero.
The span of the remaining columns is contained in $\ker\varphi$, so its
dimension drops by at least one.

Repeat this procedure until the remaining restriction is a code. If the
final span has codimension $d$, then there were at most $d$ peeling
steps, and each step removed fewer than $h$ columns. Hence
$|S_0|<dh$.
\end{proof}

\begin{lemma}[Atom bound after peeling]
\label{lem: atom bound after peeling}
Let $C:\F_p^n\to V$ be surjective, and let $(S_0,H)$ be as in
\Cref{lem: peeling noncode map}. Put
$$
r:=\dim_{\F_p}V,
\qquad
d:=\operatorname{codim}_V H.
$$
If $X\in\F_p^n$ has independent $\epsilon$-balanced coordinates, then,
for every $y\in V$,
\begin{equation}
\label{eq: peeled atom bound}
\mathbf{P}(CX=y)
\leq
(1-\epsilon)^d
\left(
p^{-(r-d)}
+
\exp\left(-\frac{\epsilon h}{p^2}\right)
\right).
\end{equation}
\end{lemma}

\begin{proof}
Let $\pi:V\to V/H$ be the quotient map. Since the columns outside
$S_0$ span $H$ and $C$ is surjective, the map
$$
\pi C|_{\F_p^{S_0}}:\F_p^{S_0}\longrightarrow V/H
$$
is surjective. Choose $d$ coordinates in $S_0$ whose projected columns
form a basis of $V/H$. After conditioning on all other coordinates in
$S_0$, the equation
$$
\pi C X_{S_0}=\pi y
$$
uniquely determines these $d$ pivot variables. Since every atom of each
coordinate is at most $1-\epsilon$, this quotient equation has
probability at most $(1-\epsilon)^d$.

Condition now on a value of $X_{S_0}$ satisfying the quotient equation.
Then
$$
y-CX_{S_0}\in H.
$$
The remaining coordinates are independent of $X_{S_0}$, and
$C|_{\F_p^{S_0^c}}$ is a code of distance at least $h$ onto $H$.
Applying \Cref{lem: Fourier equidistribution for a code} gives
$$
\mathbf{P}\left(
CX_{S_0^c}=y-CX_{S_0}
\right)
\leq
p^{-(r-d)}
+
\exp\left(-\frac{\epsilon h}{p^2}\right).
$$
Multiplying the two bounds proves the claim.
\end{proof}

\subsection{Surjections as intertwining maps}

The module structure of the target enters the estimates only through the
linear operator representing multiplication by $t$.

\begin{lemma}[Intertwining representation]
\label{lem: surjection intertwining representation}
Let $G$ be a finite $\F_p[t]$-module, let $V$ be its underlying
$\F_p$-vector space, and let
$$
T:V\longrightarrow V
$$
denote multiplication by $t$. Then, for every
$B\in\Mat_n(\F_p)$,
\begin{equation}
\label{eq: intertwining representation}
\#\Sur_{\F_p[t]}
\left(
\Cok_{\F_p[t]}(tI_n-B),G
\right)
=
\#\left\{
C:\F_p^n\twoheadrightarrow V:
CB=TC
\right\}.
\end{equation}
\end{lemma}

\begin{proof}
An $\F_p$-linear map $C:\F_p^n\to V$ defines an
$\F_p[t]$-linear map from the module presented by $tI_n-B$ precisely
when
$$
CB=TC.
$$
If this relation holds, then
$$
T(\operatorname{im}C)
=
TC(\F_p^n)
=
CB(\F_p^n)
\subseteq
\operatorname{im}C,
$$
so $\operatorname{im}C$ is automatically an $\F_p[t]$-submodule of
$G$. Thus the induced module map is surjective if and only if $C$ is
surjective as an $\F_p$-linear map.
\end{proof}

Write $X_j\in\F_p^n$ for the $j$-th column of $B_n$ and
$C_j:=C(e_j)$. Since the columns $X_1,\ldots,X_n$ are independent,
\Cref{lem: surjection intertwining representation} yields
\begin{equation}
\label{eq: expected surjection intertwining sum}
\E\left[
\#\Sur_{\F_p[t]}(M_n,G)
\right]
=
\sum_{C:\F_p^n\twoheadrightarrow V}
\prod_{j=1}^n
\mathbf{P}(CX_j=TC_j).
\end{equation}

\subsection{Proof of the growing-target estimate}

\begin{proof}[Proof of \Cref{thm: growing target surjection moments}]
The case $G=0$ is immediate, so assume
$$
r:=\dim_{\F_p}G\geq 1.
$$
Set
$$
a:=-\log(1-\epsilon)>0.
$$
Choose $\delta\in(0,1/4)$ sufficiently small that
\begin{equation}
\label{eq: choice of delta}
H(\delta)+\delta\log p
<
\min\left\{
\frac{a}{4},
\frac{\log p}{4}
\right\},
\end{equation}
where
$$
H(x):=-x\log x-(1-x)\log(1-x)
$$
is the binary entropy function. For $n$ sufficiently large, set $h:=\left\lfloor
\frac{\delta n}{r}\right\rfloor$. The finitely many remaining values of $n$ will be absorbed into the constants. Define
$$
\eta:=p^r\exp\left(-\frac{\epsilon h}{p^2}\right).
$$
Since $r\leq L\log(n+2)$ and
$$
h\geq\frac{\delta n}{r}-1,
$$
we have
$$
\log\eta\leq r\log p-\frac{\epsilon\delta}{p^2}\frac{n}{r}+\frac{\epsilon}{p^2}.
$$
Thus there exists $c_1=c_1(p,\epsilon,L)>0$ such that
\begin{equation}
\label{eq: eta bound}
\eta\leq\exp\left(-c_1\frac{n}{\log(n+2)}
\right)
\end{equation}
for all sufficiently large $n$. After decreasing $c_1$ if necessary,
the same type of estimate holds with $n\eta$ in place of $\eta$.

We split the sum in
\eqref{eq: expected surjection intertwining sum} according to whether
$C$ is a code of distance at least $h$.

Suppose first that $C:\F_p^n\twoheadrightarrow V$ is such a code. By \Cref{lem: Fourier equidistribution for a code}, for every $j$,
$$
p^{-r}-\exp\left(-\frac{\epsilon h}{p^2}\right)\leq\mathbf{P}(CX_j=TC_j)
\leq p^{-r}+\exp\left(-\frac{\epsilon h}{p^2}\right).
$$
Equivalently,
\begin{equation}
\label{eq: code column probability}
p^{-r}(1-\eta)\leq\mathbf{P}(CX_j=TC_j)
\leq p^{-r}(1+\eta).
\end{equation}

We also need that almost every linear map $\F_p^n\to V$ is both
surjective and a code of distance at least $h$. For a fixed nonzero
$\varphi\in V^*$, the number of maps $C$ satisfying
$$
\#\{i:\varphi(C_i)\neq 0\}<h
$$
is at most
$$
p^{(r-1)n}
\sum_{s<h}\binom{n}{s}p^s.
$$
Taking a union bound over the $p^r-1$ nonzero functionals and using
$h/n\leq\delta/r\leq\delta$, the second inequality in
\eqref{eq: choice of delta} gives
\begin{equation}
\label{eq: noncode proportion}
\frac{
\#\{C:\F_p^n\to V:C\text{ is not a code of distance }h\}
}{
p^{rn}
}
\leq
e^{-c_2n}
\end{equation}
for some $c_2>0$. Similarly, the proportion of non-surjective maps is
at most
$$
(p^r-1)p^{-n}
=
e^{-\Omega(n)},
$$
since $r=O(\log n)$. Hence the number of surjective code maps equals
$$
p^{rn}\left(1-e^{-\Omega(n)}\right).
$$
Combining this with \eqref{eq: code column probability}, their total
contribution to \eqref{eq: expected surjection intertwining sum} lies
between
$$
\left(1-e^{-\Omega(n)}\right)(1-\eta)^n
$$
and $(1+\eta)^n$. By \eqref{eq: eta bound}, the code contribution is therefore
\begin{equation}
\label{eq: code contribution}
1+O\left(\exp\left(-c_3\frac{n}{\log(n+2)}
\right)\right)
\end{equation}
for some $c_3>0$.

It remains to control the non-code maps. Let
$C:\F_p^n\twoheadrightarrow V$ be a surjective map which is not a code
of distance $h$, and choose one peeling pair $(S_0,H)$ supplied by
\Cref{lem: peeling noncode map}. Put
$$
d:=\operatorname{codim}_V H\geq 1,
\qquad
s:=|S_0|<dh.
$$
For fixed $d,s,S_0,H$, the number of maps compatible with the coarse
condition
$$
C_i\in H
\qquad
\text{for every }i\notin S_0
$$
is at most
$$
p^{(r-d)(n-s)}p^{rs}.
$$
The number of codimension-$d$ subspaces $H\leq V$ is at most
\begin{equation}
\label{eq: gaussian binomial bound}
C_p p^{d(r-d)},
\qquad C_p:=\prod_{j\geq 1}(1-p^{-j})^{-1}.
\end{equation}
Thus the number of maps having a chosen peeling pair with parameters
$(d,s)$ is at most
\begin{equation}
\label{eq: peeling map count}
C_p\binom{n}{s}p^{d(r-d)}p^{(r-d)(n-s)+rs}.
\end{equation}
This deliberately overcounts the admissible maps, which is harmless.

For every genuinely admissible $C$,
\Cref{lem: atom bound after peeling} gives, for each column $X_j$,
$$
\begin{aligned}
\mathbf{P}(CX_j=TC_j)
&\leq
(1-\epsilon)^d
\left(
p^{-(r-d)}
+
\exp\left(-\frac{\epsilon h}{p^2}\right)
\right)\\
&\leq
(1-\epsilon)^d
p^{-(r-d)}(1+\eta).
\end{aligned}
$$
Combining this with \eqref{eq: peeling map count}, the total contribution
for fixed $d$ and $s$ is at most
\begin{equation}
\label{eq: fixed ds noncode contribution}
C_p
\binom{n}{s}
p^{d(r-d)+ds}
(1-\epsilon)^{dn}
(1+\eta)^n.
\end{equation}
The factor $p^{(r-d)n}$ coming from the number of maps has cancelled
exactly with the factor $p^{-(r-d)n}$ coming from the probability
estimate.

We now sum over $s<dh$. Since
$$
\frac{dh}{n}
\leq
\frac{\delta d}{r}
\leq
\delta,
$$
we have
$$
\sum_{s<dh}
\binom{n}{s}p^{ds}
\leq
(n+1)
\exp\left(
nH\left(\frac{\delta d}{r}\right)
+
\frac{\delta d^2}{r}n\log p
\right).
$$
Because $\delta d/r\leq\delta<1/2$ and $d\geq 1$,
$$
H\left(\frac{\delta d}{r}\right)
\leq
H(\delta)
\leq
dH(\delta),
$$
while
$$
\frac{\delta d^2}{r}\log p
\leq
\delta d\log p.
$$
Hence \eqref{eq: choice of delta} implies
\begin{equation}
\label{eq: entropy peeling bound}
\sum_{s<dh}
\binom{n}{s}p^{ds}
\leq
\exp\left(
\frac{ad}{4}n+O(\log n)
\right).
\end{equation}
Furthermore,
$$
p^{d(r-d)}
=
e^{O(dr)}
=
e^{O(d\log n)}
$$
and, by \eqref{eq: eta bound},
$$
(1+\eta)^n
\leq
e^{n\eta}
=
e^{o(n)}.
$$
Since $(1-\epsilon)^{dn}=e^{-adn}$, summing \eqref{eq: fixed ds noncode contribution} over $s<dh$ and using
\eqref{eq: entropy peeling bound} shows that the contribution of maps
with a fixed codimension $d$ is at most
$e^{-c_4dn}$ for some $c_4>0$ and all sufficiently large $n$. Summing over
$1\leq d\leq r$ gives
\begin{equation}
\label{eq: noncode contribution}
\sum_{\substack{
C:\F_p^n\twoheadrightarrow V\\
C\text{ not a code of distance }h
}}
\prod_{j=1}^n
\mathbf{P}(CX_j=TC_j)
\leq
e^{-c_5n}
\end{equation}
for some $c_5>0$.

Combining \eqref{eq: code contribution} and
\eqref{eq: noncode contribution} proves
\eqref{eq: growing target surjection moment}. The uniform boundedness
in \eqref{eq: uniformly bounded growing surjection moments} follows by
absorbing the finitely many smaller values of $n$ into the constant.
\end{proof}

\begin{remark}
The proof above does not use the isomorphism type of $G$ beyond the
dimension of its underlying $\F_p$-vector space. In particular, the
estimate is uniform over all $\F_p[t]$-module structures on targets of
dimension $O(\log n)$. Moreover, the entries of $B_n$ need not be
identically distributed: independence and the common
$\epsilon$-balanced condition are sufficient.
\end{remark}

\subsection{Tail bounds for primary partitions}

We now apply \Cref{thm: growing target surjection moments} to the primary
partitions of $B_n$. Fix $i\geq 1$, and write
$$
\lambda^{(i)}(B_n)
=
\left(
\lambda^{(i)}_1,\lambda^{(i)}_2,\ldots
\right).
$$
Set
$$
a_i:=\lambda^{(i)}_1,
\qquad
\ell_i:=\#\{j:\lambda^{(i)}_j>0\}.
$$
Then
\begin{equation}
\label{eq: partition size largest times length}
\left|\lambda^{(i)}(B_n)\right|
\leq
a_i\ell_i.
\end{equation}

We first control the largest part. Let
$$
G_{i,k}:=\F_p[t]/(F_i^k).
$$
Then $\dim_{\F_p}G_{i,k}=d_i k$, and
\begin{equation}
\label{eq: largest part surjection criterion}
a_i\geq k
\quad\Longleftrightarrow\quad
\Sur_{\F_p[t]}(M_n,G_{i,k})\neq\varnothing.
\end{equation}
Once one surjection exists, postcomposition by automorphisms of
$G_{i,k}$ produces at least
$$
\#\Aut_{\F_p[t]}(G_{i,k})=(p^{d_i}-1)p^{d_i(k-1)}
$$
distinct surjections. Therefore, whenever
$d_i k\le L\log(n+2)$,
\Cref{thm: growing target surjection moments} and Markov's inequality
give
\begin{equation}
\label{eq: largest primary part tail}
\mathbf{P}(a_i\geq k)\leq C_{i,L} p^{-d_i k}.
\end{equation}

We next control the number of parts. Let
$$
K_i:=\F_p[t]/(F_i)\cong\F_{p^{d_i}}.
$$
Only the $F_i$-primary component of $M_n$ maps nontrivially to $K_i$,
and
$$
\Hom_{\F_p[t]}(M_n,K_i)
\cong
K_i^{\ell_i}.
$$
Since $K_i$ is simple, every nonzero homomorphism to $K_i$ is
surjective. Hence
$$
\#\Sur_{\F_p[t]}(M_n,K_i)
=
p^{d_i\ell_i}-1.
$$
Applying the fixed-target case of
\Cref{thm: growing target surjection moments} and Markov's inequality
gives, uniformly in $n$,
\begin{equation}
\label{eq: number primary blocks tail}
\mathbf{P}(\ell_i\geq u)
\leq
C_i p^{-d_i u}.
\end{equation}

Combining these two estimates yields a tail bound for the total
$F_i$-primary degree.

\begin{lemma}
\label{lem: one primary degree tail}
Fix $i\geq 1$ and $M>0$. There exist constants
$C_{i,M},c_{i,M}>0$ such that, for every $D\geq 2$,
\begin{equation}
\label{eq: one primary degree tail}
\sup_{n\geq 1}
\mathbf{P}\left(
d_i\left|\lambda^{(i)}(B_n)\right|>D
\right)
\leq
C_{i,M}D^{-M}
+
C_{i,M}
\exp\left(
-c_{i,M}\frac{D}{\log(D+2)}
\right).
\end{equation}
\end{lemma}

\begin{proof}
Set
$$
k_i(D):=
\left\lceil
\frac{M}{d_i}\log_p(D+2)
\right\rceil.
$$
If $n\leq D$, then
$$
d_i\left|\lambda^{(i)}(B_n)\right|
\leq n\leq D,
$$
so the event in \eqref{eq: one primary degree tail} is impossible.

Suppose therefore that $n>D$. Then
$$
d_i k_i(D)=O_{i,M}(\log n),
$$
so \eqref{eq: largest primary part tail} applies. By
\eqref{eq: partition size largest times length}, the event
$$
d_i\left|\lambda^{(i)}(B_n)\right|>D
$$
implies either $a_i\geq k_i(D)$ or $\ell_i>\frac{D}{d_i k_i(D)}$. Consequently, \eqref{eq: largest primary part tail} and
\eqref{eq: number primary blocks tail} give
$$
\mathbf{P}\left(
d_i\left|\lambda^{(i)}(B_n)\right|>D
\right)
\leq
C_{i,M}D^{-M}
+
C_{i,M}
\exp\left(
-c_{i,M}\frac{D}{\log(D+2)}
\right),
$$
uniformly in $n$.
\end{proof}

\subsection{Moment tightness for distinguished degrees}

We first deduce the quantitative tail estimate from the bounds for the
individual primary components.

\begin{proof}[Proof of \Cref{prop: quantitative degree tail}]
By \eqref{eq: degree from primary partitions},
$$
\deg P_{A_n}^{S}=\sum_{i\in S}
d_i\left|\lambda^{(i)}(B_n)\right|.
$$
Hence the event $\deg P_{A_n}^{S}>D$ implies that, for some $i\in S$,
$$
d_i\left|\lambda^{(i)}(B_n)\right|>\frac{D}{|S|}.
$$
Therefore, by a union bound,
$$
\mathbf{P}\left(\deg P_{A_n}^{S}>D
\right)\leq\sum_{i\in S}\mathbf{P}\left(
d_i\left|\lambda^{(i)}(B_n)\right|
>\frac{D}{|S|}\right).
$$
Applying \Cref{lem: one primary degree tail} to each $i\in S$ and
absorbing the fixed factor $|S|$ into the constants gives
\eqref{eq: quantitative degree tail}.
\end{proof}

We can now deduce the desired moment tightness.

\begin{proof}[Proof of \Cref{prop: degree moment tightness}]
Fix an integer $m\geq 1$, and set
$$
X_n:=\deg P_{A_n}^{S}.
$$
For every $D>0$,
$$
X_n^m\mathbf{1}_{\{X_n>D\}}=D^m\mathbf{1}_{\{X_n>D\}}+\mathbf{1}_{\{X_n>D\}}
\int_D^{X_n}mt^{m-1}\,dt.
$$
Taking expectations and applying Tonelli's theorem gives
$$
\E\left[X_n^m\mathbf{1}_{\{X_n>D\}}
\right]=D^m\mathbf{P}(X_n>D)+\int_D^\infty
mt^{m-1}\mathbf{P}(X_n>t)\,dt.
$$
Apply \Cref{prop: quantitative degree tail} with some $M>m+2$. We obtain
$$
\begin{aligned}
\E\left[X_n^m\mathbf{1}_{\{X_n>D\}}
\right]
\leq {}&
C D^{m-M}+Cm\int_D^\infty t^{m-1-M}\,dt \\
&+C D^m\exp\left(-c\frac{D}{\log(D+2)}
\right) \\
&+Cm\int_D^\infty
t^{m-1}\exp\left(-c\frac{t}{\log(t+2)}
\right)\,dt,
\end{aligned}
$$
where the constants are independent of $n$. Since $M>m+2$, the first
two terms tend to zero as $D\to\infty$. The final two terms also tend to
zero, since
$$
\exp\left(
-c\frac{t}{\log(t+2)}
\right)
$$
decays faster than any fixed negative power of $t$. Consequently,
$$
\lim_{D\to\infty}
\sup_{n\geq 1}
\E\left[
\left(\deg P_{A_n}^{S}\right)^m
\mathbf{1}_{\{\deg P_{A_n}^{S}>D\}}
\right]
=0.
$$
This proves \Cref{prop: degree moment tightness}.
\end{proof}

%% file: The_random_polynomial_case.tex
\section{The random polynomial case}
\label{sec: random polynomial}

In this section, we prove
\Cref{thm: root correlation universality} by comparing the random
polynomial model under consideration with the corresponding additive Haar
coefficient model. For the rest of this section, for every
$n\geq 1$, let $P_n$ be as in \eqref{item: P_n}, and let
$$
P_n^{\Haar}(x):=\eta_nx^n+\cdots+\eta_1x+\eta_0
$$
be the random Haar polynomial defined in
\eqref{eq: Haar random polynomial}. We first show that, for every fixed test
polynomial whose constant coefficient is a unit, the resultant valuations
associated with $P_n$ and $P_n^{\Haar}$ converge weakly to the same limiting
random variable. We then establish the moment tightness of the degrees of the
relevant distinguished factors. Combining these two inputs with the resultant
distribution method developed in
\Cref{sec: resultant distribution method}, we prove the universality of the
expected root statistics. 

\begin{thm}[Universality of resultant distributions]
\label{thm: random polynomial resultant convergence}
Let $Z\in\Z_p[x]$ be a fixed monic polynomial of degree $d$, whose constant term lies in $\Z_p^\times$. Then there exists a random variable $\mathcal{R}_Z\in\Z_{\geq 0}\cup\{\infty\}$ such that
$$
\val(\Res(P_n,Z))
\overset{d}{\longrightarrow}
\mathcal{R}_Z,\qquad n\rightarrow\infty.
$$
Furthermore, the sequence $\val(\Res(P_n^{\Haar},Z))$ also weakly converges to the same limit, i.e.,
$$
\val(\Res(P_n^{\Haar},Z))
\overset{d}{\longrightarrow}
\mathcal{R}_Z,\qquad n\rightarrow\infty.
$$
\end{thm}

\begin{proof}
The case $d=0,Z=1$ is trivial. From now on, we assume $d\ge 1$. For every integer $k\geq 1$, consider the finite additive group
$$
\Z_p[x]/(Z,p^k)\cong(\Z/p^k\Z)[x]/(Z),
$$
which has order $p^{kd}$. We first show that the image of $P_n$ in this group converges weakly to the uniform distribution.

Let $\chi$ be a nontrivial additive character of $(\Z/p^k\Z)[x]/(Z)$. We continue to write $x$ for the residue class of $x$ in this quotient. By the independence of the coefficients,
\begin{equation}
\label{eq: Fourier transform random polynomial quotient}
\E\left[
\chi\left(\sum_{j=0}^n \xi_jx^j\right)
\right]
=
\prod_{j=0}^n
\E\left[\chi(\xi_jx^j)\right].
\end{equation}
For each $j\geq 0$, define an additive character of $\Z/p^k\Z$ by
$$
\chi_j(a):=\chi(ax^j).
$$

Since the constant term of $Z$ lies in $\Z_p^\times$, the residue class of
$x$ is a unit in $(\Z/p^k\Z)[x]/(Z)$. Moreover,
$$
1,x,\ldots,x^{d-1}
$$
form a basis of this quotient as a $\Z/p^k\Z$-module. It follows that, for
every $r\geq 0$,
$$
x^r,x^{r+1},\ldots,x^{r+d-1}
$$
also form a basis. Consequently, among every $d$ consecutive characters
$$
\chi_r,\chi_{r+1},\ldots,\chi_{r+d-1},
$$
at least one is nontrivial. Indeed, if all of them were trivial, then $\chi$
would vanish on a basis and hence would be the trivial character.

We next record the Fourier decay following from the $\epsilon$-balanced
assumption. For every fixed $k\geq 1$, there exists a constant
$$
\rho_k=\rho_k(p,\epsilon)<1
$$
such that, for every $\epsilon$-balanced random variable $\xi\in\Z_p$ and
every nontrivial additive character
$\psi:\Z/p^k\Z\to\mathbb{C}^\times$, one has
\begin{equation}
\label{eq: balanced Fourier decay}
\left|\E[\psi(\xi)]\right|\leq \rho_k.
\end{equation}
Indeed, the law of $\xi\bmod p^k$ belongs to the compact set of probability
measures $\mu$ on $\Z/p^k\Z$ satisfying
$$
\max_{a\in\Z/p\Z}
\mu(a+p\Z/p^k\Z)
\leq 1-\epsilon.
$$
For a fixed nontrivial character $\psi$, the modulus of
$$
\sum_{a\in\Z/p^k\Z}\mu(a)\psi(a)
$$
can equal one only if $\psi$ is constant on the support of $\mu$. Since the
support of $\mu$ meets at least two distinct residue classes modulo $p$, this
is impossible for a nontrivial additive character. The compactness of the
set of admissible measures and the finiteness of the collection of
characters give \eqref{eq: balanced Fourier decay}.

Applying \eqref{eq: balanced Fourier decay} to
\eqref{eq: Fourier transform random polynomial quotient}, and using the fact
that every block of $d$ consecutive characters contains a nontrivial one, we
obtain
$$
\left|
\E\left[
\chi\left(\sum_{j=0}^n \xi_jx^j\right)
\right]
\right|
\leq
\rho_k^{\lfloor (n+1)/d\rfloor}.
$$
Hence, for every nontrivial additive character $\chi$,
$$
\E\left[
\chi\left(P_n\bmod (Z,p^k)\right)
\right]
\longrightarrow 0.
$$
By Fourier inversion on the finite abelian group
$(\Z/p^k\Z)[x]/(Z)$, it follows that
\begin{equation}
\label{eq: random polynomial quotient uniform convergence}
P_n\bmod (Z,p^k)
\overset{d}{\longrightarrow}
\mathbf{H}_{Z,k},
\end{equation}
where $\mathbf{H}_{Z,k}$ denotes a uniformly distributed random element of
$(\Z/p^k\Z)[x]/(Z)$.

We now compare this with the Haar coefficient model. Since $1,x,\ldots,x^{d-1}$ form a $\Z/p^k\Z$-basis of $(\Z/p^k\Z)[x]/(Z)$, the random element
$$
\eta_{d-1}x^{d-1}+\cdots+\eta_1x+\eta_0
\pmod{(Z,p^k)}
$$
is uniformly distributed on $(\Z/p^k\Z)[x]/(Z)$. Adding the remaining
independent terms preserves the uniform distribution. Therefore, for every
$n\geq d-1$,
\begin{equation}
\label{eq: Haar polynomial quotient uniform}
P_n^{\Haar}\bmod (Z,p^k)
\overset{d}{=}
\mathbf{H}_{Z,k}.
\end{equation}

We next show that the truncated resultant valuation depends only on the
residue class modulo $(Z,p^k)$. Suppose that
$$
Q_1\equiv Q_2\pmod{(Z,p^k)}.
$$
Then there exist polynomials $B,C\in\Z_p[x]$ such that
$$
Q_1-Q_2=BZ+p^kC.
$$
By the fourth item of \Cref{prop: properties of resultant},
$$
\Res(Z,Q_1)=\Res(Z,Q_2+p^kC).
$$
Since the resultant is a polynomial with integer coefficients in the coefficients of its two arguments, we have
$$
\Res(Z,Q_2+p^kC)\equiv\Res(Z,Q_2)
\pmod{p^k}.
$$
It follows from the first item of \Cref{prop: properties of resultant} that
\begin{align}
\begin{split}
\min\{\val(\Res(Q_1,Z)),k\}&=\min\{\val(\Res(Z,Q_1)),k\}\\
&=\min\{\val(\Res(Z,Q_2)),k\}\\
&=\min\{\val(\Res(Q_2,Z)),k\}.
\end{split}
\end{align}
Thus there is a well-defined function
$$
\Phi_{Z,k}:
(\Z/p^k\Z)[x]/(Z)
\longrightarrow
\{0,1,\ldots,k\}
$$
given by
$$
\Phi_{Z,k}(Q)
:=
\min\{\val\Res(\widetilde Q,Z),k\},
$$
where $\widetilde Q\in\Z_p[x]$ is any lift of $Q$.

Applying $\Phi_{Z,k}$ to
\eqref{eq: random polynomial quotient uniform convergence}, we obtain
$$
\min\{\val(\Res(P_n,Z)),k\}
\overset{d}{\longrightarrow}
\Phi_{Z,k}(\mathbf{H}_{Z,k}).
$$
On the other hand, by \eqref{eq: Haar polynomial quotient uniform},
$$
\min\{\val(\Res(P_n^{\Haar},Z)),k\}
\overset{d}{=}
\Phi_{Z,k}(\mathbf{H}_{Z,k})
$$
for every $n\geq d-1$. Therefore, for every fixed $k\geq 1$, the two sequences of truncated resultant valuations converge to the same distribution.

Now define $\mathcal{R}_Z:=\val(\Res(P_{d-1}^{\Haar},Z))$. For every $n\geq d-1$, the image of $P_n^{\Haar}$ modulo $Z$ has the same
Haar distribution as the image of $P_{d-1}^{\Haar}$ modulo $Z$. Hence
$$
\val(\Res(P_n^{\Haar},Z))
\overset{d}{=}\mathcal{R}_Z.
$$
Moreover, for every fixed $k\geq 1$,
$$
\min\{\val(\Res(P_n,Z)),k\}
\overset{d}{\longrightarrow}
\min\{\mathcal{R}_Z,k\}.
$$
It follows that, for every $r\in\Z_{\geq 0}$, upon choosing $k>r$,
$$
\mathbf{P}\left(\val(\Res(P_n,Z))=r\right)
\longrightarrow
\mathbf{P}(\mathcal{R}_Z=r).
$$
Also, for every $k\geq 1$,
$$
\mathbf{P}\left(\val(\Res(P_n,Z))\geq k\right)
\longrightarrow
\mathbf{P}(\mathcal{R}_Z\geq k).
$$
These relations imply
$$
\val(\Res(P_n,Z))\overset{d}{\longrightarrow}
\mathcal{R}_Z
$$
in $\Z_{\geq 0}\cup\{\infty\}$. Together with
$$
\val(\Res(P_n^{\Haar},Z))
\overset{d}{=}
\mathcal{R}_Z,
\qquad n\geq d-1,
$$
this proves the theorem.
\end{proof}

\begin{rmk}
\label{rem: Haar resultant stabilization}
In fact, the above proof shows that the convergence statement for the Haar coefficient model is stronger than what is stated in
\Cref{thm: random polynomial resultant convergence}. Indeed, for every $n\geq d-1$, the image of $P_n^{\Haar}$ in $\Z_p[x]/(Z)$ is Haar distributed. Consequently,
$$
\val(\Res(P_n^{\Haar},Z))\overset{d}{=}
\val(\Res(P_{d-1}^{\Haar},Z))=\mathcal{R}_Z,
\qquad n\geq d-1.
$$
Thus the resultant-valuation distribution in the Haar coefficient model stabilizes exactly once the degree reaches $d-1$.
\end{rmk}

\begin{prop}[Moment tightness of the lifted-subspace factor]
\label{prop: degree moment tightness random polynomial}
Fix an integer $m\geq 1$, and let $S={i_1,\ldots,i_s}\subset\Z_{>0}\setminus{1}$ be finite. Let $P_n^{S},P_n^{\Haar,S}$ be the $S$-distinguished factor of $P_n,P_n^{\Haar}$, respectively. Then
$$
\lim_{D\to\infty}\sup_{n\geq 1}\E\left[(\deg P_n^S)^m
\mathbf{1}_{{\deg P_n^S>D}}
\right]=\lim_{D\to\infty}
\sup_{n\geq 1}\E\left[
(\deg P_n^{\Haar,S})^m\mathbf{1}_{{\deg P_n^{\Haar,S}>D}}\right]=0.
$$
\end{prop}

\begin{proof}
Recall that $F_i\in\F_p[x]$ denotes the monic irreducible polynomial corresponding to the lifted subspace $\mathcal{U}_i$. Following the notation in \eqref{subsec:zpbar}, we write $d_i:=\deg F_i$. Since $1\notin S$, none of the polynomials $F_i$, $i\in S$, is equal to $x$.
Let
$$
\bar P_n(x):=\bar\xi_nx^n+\cdots+\bar\xi_1x+\bar\xi_0\in\F_p[x]
$$
be the reduction of $P_n$ modulo $p$, and set the random event
$$
\mathcal{E}_n:=\{\bar P_n=0\}=
\{\xi_0,\ldots,\xi_n\in p\Z_p\}.
$$
By independence and the $\epsilon$-balanced assumption,
\begin{equation}
\label{eq: all coefficients divisible by p}
\mathbf{P}(\mathcal{E}_n)
\leq
(1-\epsilon)^{n+1}.
\end{equation}

For a given $(r_i)_{i\in S}\in\Z_{\geq 0}^S$, we define
$$
Q_{\mathbf r}(x):=\prod_{i\in S}F_i(x)^{r_i},
\qquad L_{\mathbf r}:=\deg Q_{\mathbf r}=
\sum_{i\in S}d_ir_i.
$$
We first claim that, whenever $L_{\mathbf r}\leq n+1$,
\begin{equation}
\label{eq: finite field polynomial divisibility bound}
\mathbf{P}\left(
Q_{\mathbf r}\mid\bar P_n
\right)
\leq
(1-\epsilon)^{L_{\mathbf r}}.
\end{equation}

To prove the claim, condition on the coefficients $\bar\xi_{L_{\mathbf r}},
\ldots,\bar\xi_n$. Modulo $Q_{\mathbf r}$, the divisibility condition
$Q_{\mathbf r}\mid\bar P_n$ uniquely determines the vector
$$
(\bar\xi_0,\ldots,\bar\xi_{L_{\mathbf r}-1})
\in\F_p^{L_{\mathbf r}}.
$$
Indeed, $1,x,\ldots,x^{L_{\mathbf r}-1}$ form an $\F_p$-basis of $\F_p[x]/(Q_{\mathbf r})$. By independence and the $\epsilon$-balanced assumption, the probability that the first
$L_{\mathbf r}$ coefficients equal any prescribed vector is at most
$(1-\epsilon)^{L_{\mathbf r}}$. Averaging over the remaining coefficients
proves \eqref{eq: finite field polynomial divisibility bound}.

Set $d_{\max}:=\max_{i\in S}d_i$. We next establish an exponential tail bound for $\deg P_n^S$ on the event $\mathcal{E}_n^c$. On this event, define
$$
M_{i,n}:=\ord_{F_i}(\bar P_n),
\qquad i\in S.
$$
By the strong form of Hensel's lemma given in \Cref{lem: Hensel},
$$
\deg P_n^S=\sum_{i\in S}d_iM_{i,n}.
$$
If $\deg P_n^S\geq R$, we may choose integers
$$
0\leq r_i\leq M_{i,n},
\qquad i\in S,
$$
such that $R\leq L_{\mathbf r}<R+d_{\max}$. Indeed, starting from $\mathbf r=0$, increase the coordinates one at a time,
without exceeding $M_{i,n}$, and stop when $L_{\mathbf r}$ first reaches or
exceeds $R$. At the last step, $L_{\mathbf r}$ increases by at most $d_{\max}$. For this choice of $\mathbf r$, we have $Q_{\mathbf r}\mid\bar P_n$. It follows that
\begin{equation}
\label{eq: lifted degree event union}
{\deg P_n^S\geq R}\cap\mathcal{E}_n^c
\subset
\bigcup_{\substack{\mathbf r\in\Z_{\geq 0}^S\
R\leq L_{\mathbf r}<R+d_{\max}}}
{Q_{\mathbf r}\mid\bar P_n}.
\end{equation}
Moreover, for every vector $\mathbf r$ appearing in
\eqref{eq: lifted degree event union},
\eqref{eq: finite field polynomial divisibility bound} gives
$$
\mathbf{P}(Q_{\mathbf r}\mid\bar P_n)
\leq
(1-\epsilon)^{L_{\mathbf r}}
\leq
(1-\epsilon)^R.
$$

For each $\ell\geq 0$, the number of vectors
$\mathbf r\in\Z_{\geq 0}^S$ satisfying $L_{\mathbf r}=\ell$ is bounded by $O_S((\ell+1)^{s-1})$. Since the interval
$$
R\leq L_{\mathbf r}<R+d_{\max}
$$
contains only $d_{\max}$ possible values of $L_{\mathbf r}$, the union in \eqref{eq: lifted degree event union} contains
$$
O_S((R+d_{\max}+1)^{s-1})
$$
events. Therefore,
\begin{equation}
\label{eq: lifted degree exponential tail}
\sup_{n\geq 1}\mathbf{P}\left(
\deg P_n^S\geq R,\mathcal{E}_n^c\right)
\leq O_S((R+d_{\max}+1)^{s-1})(1-\epsilon)^R\leq Ce^{-cR},
\qquad \forall R\geq 0.
\end{equation}
Here $C=C(S,\epsilon),c=c(S,\epsilon)>0$ depend only on $S$ and $\epsilon$.

We now deduce the required moment tightness. For every nonnegative integer-valued random variable $Y$ and every $D\geq 1$, we have
$$
Y^m\mathbf{1}_{{Y>D}}=(D+1)^m\mathbf{1}_{{Y>D}}+\sum_{r>D}\bigl((r+1)^m-r^m\bigr)\mathbf{1}_{{Y>r}}.
$$
Applying this identity with $Y=\deg P_n^S$ on the event $\mathcal{E}_n^c$ and using \eqref{eq: lifted degree exponential tail}, we obtain
\begin{equation}\label{eq: lifted degree moment tail}
\sup_{n\geq 1}\E\left[
(\deg P_n^S)^m\mathbf{1}_{{\deg P_n^S>D}}
\mathbf{1}_{\mathcal{E}_n^c}\right]\leq C(D+1)^m e^{-cD}+C\sum_{r>D}
\bigl((r+1)^m-r^m\bigr)e^{-cr}.
\end{equation}
The right-hand side tends to zero as $D\to\infty$.

It remains to control the event $\mathcal{E}_n$. Since $\deg P_n^S\leq n$, we have
$$
\E\left[(\deg P_n^S)^m
\mathbf{1}_{{\deg P_n^S>D}}
\mathbf{1}_{\mathcal{E}_n}
\right]\leq n^m\mathbf{1}_{{n>D}}\mathbf{P}(\mathcal{E}_n)\leq n^m\mathbf{1}_{{n>D}}(1-\epsilon)^{n+1}.
$$
Consequently,
$$
\sup_{n\geq 1}
\E\left[
(\deg P_n^S)^m
\mathbf{1}_{{\deg P_n^S>D}}
\mathbf{1}_{\mathcal{E}_n}
\right]\leq
\sup_{n>D}n^m(1-\epsilon)^{n+1},
$$
which also tends to zero as $D\to\infty$. Combining this estimate with
\eqref{eq: lifted degree moment tail} gives
$$
\lim_{D\to\infty}
\sup_{n\geq 1}
\E\left[
(\deg P_n^S)^m
\mathbf{1}_{{\deg P_n^S>D}}
\right]
=0.
$$

The assertion for the Haar coefficient model follows in exactly the same
way. In this case,
$$
\mathbf{P}\left(
\eta_0,\ldots,\eta_n\in p\Z_p
\right)=p^{-(n+1)},
$$
and the coefficients of $P_n^{\Haar}$ are independent and $(1-1/p)$-balanced.
\end{proof}

It is worth mentioning that the above proof does not require a more precise estimate obtained through Fourier analysis. For our purpose, it is enough that the exponential decay $(1-\epsilon)^D$ dominates any fixed polynomial growth in $D$.

We have now verified the resultant distribution and moment tightness
assumptions required by the resultant distribution method developed in
\Cref{sec: resultant distribution method}. To apply \Cref{thm: expected root statistics convergence}, we will need to verify that the limiting Haar distinguished factor is almost surely squarefree. For this purpose, we first identify the law of a Haar distinguished factor on each fixed-degree component. It is enough to consider a single lifted subspace, which is the case needed below.

\begin{lemma}[The Haar law on a fixed-degree distinguished-factor component]
\label{lem: Haar distinguished factor fixed degree law}
Fix $i\in\Z_{>0}$ and $r\geq 0$. Let
$$
\mathcal{C}_{i,r}
:=
\left\{
Q\in\Z_p[x]:
Q\text{ is monic},
\deg Q=rd_i,
\bar Q=F_i^r
\right\},
$$
and let $\mu_{i,r}$ denote the normalized Haar probability measure on
$\mathcal{C}_{i,r}$ under the natural coefficient identification
$$
\mathcal{C}_{i,r}\cong F_i^r+p\Z_p^{rd_i}.
$$
Then, for every $n\geq 1$, the restriction of the law of
$P_n^{\Haar,\{i\}}$ to $\mathcal{C}_{i,r}$ on the event
${\bar P_n^{\Haar}\neq 0}$ is a constant multiple of $\mu_{i,r}$.
More precisely, for every Borel set $\mathcal{A}\subset\mathcal{C}_{i,r}$,
$$
\mathbf{P}\left(
P_n^{\Haar,\{i\}}\in A,
\bar P_n^{\Haar}\neq 0
\right)=
\mathbf{P}\left(
P_n^{\Haar,\{i\}}\in\mathcal{C}_{i,r},
\bar P_n^{\Haar}\neq 0
\right)
\mu_{i,r}(\mathcal{A}).
$$
\end{lemma}

\begin{proof}
The case $r=0$ is immediate, since $\mathcal{C}_{i,0}={1}$. We therefore assume that $r\geq 1$ and $rd_i\le n$. Fix a nonzero polynomial $\bar P\in\F_p[x]$ of degree at most $n$ satisfying $\ord_{F_i}(\bar P)=r$. Write
$$
\bar P=F_i^r\bar R,
\qquad
\gcd(F_i,\bar R)=1.
$$
We regard $\bar R$ as an element of the space of polynomials over
$\F_p$ of degree at most $n-rd_i$. Define
$$
\mathcal{R}_{\bar R}
:=\left\{R\in\Z_p[x]:
\deg R\leq n-rd_i,\ \bar R=R\bmod p
\right\}
$$
and
$$
\mathcal{P}_{\bar P}
:=
\left\{
P\in\Z_p[x]:
\deg P\leq n,\ \bar P=P\bmod p
\right\}.
$$

By the strong form of Hensel's lemma given in \Cref{lem: Hensel}, every
$P\in\mathcal{P}_{\bar P}$ admits a unique factorization $P=QR$, where $Q\in\mathcal{C}_{i,r}$ and $R\in\mathcal{R}_{\bar R}$. Thus the multiplication map
$$
\Phi:
\mathcal{C}_{i,r}\times\mathcal{R}_{\bar R}
\longrightarrow
\mathcal{P}_{\bar P},
\qquad
(Q,R)\longmapsto QR,
$$
is a bijection.

We next show that $\Phi$ preserves Haar measure. At a point $(Q,R)$, its
differential is the $\Z_p$-linear map
$$
D\Phi_{(Q,R)}(\dot Q,\dot R)
=\dot Q R+Q\dot R,
$$
where
$$
\deg\dot Q<rd_i,
\qquad
\deg\dot R\leq n-rd_i.
$$
To compute its determinant, decompose the target coefficient space into
the quotient modulo $Q$ and the subspace of multiples of $Q$. Modulo
$Q$, the differential is
$$
\dot Q\longmapsto \dot Q R
\qquad
\text{in }\Z_p[x]/(Q).
$$
The determinant of multiplication by $R$ on the free $\Z_p$-module
$\Z_p[x]/(Q)$ is, up to sign,
$$
\Res(Q,R).
$$
On the other hand, the map $\dot R\longmapsto Q\dot R$ identifies the coefficient space of polynomials of degree at most $n-rd_i$ with the space of multiples of $Q$ of degree at most $n$. Since
$Q$ is monic, this identification has determinant $1$ with respect to
the natural coefficient bases. Consequently,
$$
\det D\Phi_{(Q,R)}=\pm\Res(Q,R).
$$

Since
$$
\bar Q=F_i^r
\qquad\text{and}\qquad
\gcd(F_i,\bar R)=1,
$$
the reductions $\bar Q$ and $\bar R$ are relatively prime. Therefore,
by \Cref{prop: relatively prime resultant},
$$
\val\bigl(\Res(Q,R)\bigr)=0.
$$
Hence the Jacobian determinant of $\Phi$ is a unit at every point of
$\mathcal{C}_{i,r}\times\mathcal{R}_{\bar R}$.

It follows from the $p$-adic change-of-variables formula that $\Phi$ is
measure preserving with respect to the normalized Haar measures on
these coefficient cosets. Conditional on
$$
\bar P_n^{\Haar}=\bar P,
$$
the polynomial $P_n^{\Haar}$ is Haar distributed on
$\mathcal{P}_{\bar P}$. Therefore, under the factorization
$$
P_n^{\Haar}=P_n^{\Haar,\{i\}}R,
$$
the factor $P_n^{\Haar,\{i\}}$ is Haar distributed on
$\mathcal{C}_{i,r}$; that is, its conditional distribution is
$\mu_{i,r}$.

This conditional distribution is independent of the particular
nonzero polynomial $\bar P$ satisfying
$$
\ord_{F_i}(\bar P)=r.
$$
Averaging over all such reductions, we obtain, for every Borel set
$\mathcal{A}\subset\mathcal{C}_{i,r}$,
$$
\mathbf{P}\left(
P_n^{\Haar,\{i\}}\in\mathcal{A},
\ \bar P_n^{\Haar}\neq0
\right)=\mathbf{P}\left(
P_n^{\Haar,\{i\}}\in\mathcal{C}_{i,r},
\ \bar P_n^{\Haar}\neq0
\right)
\mu_{i,r}(\mathcal{A}),
$$
as claimed.
\end{proof}

We are therefore ready to prove
\Cref{thm: root correlation universality}.

\begin{proof}[Proof of \Cref{thm: root correlation universality}]
Take a finite set $S\subset\Z_{>0}\setminus{1}$ such that
$$
\mathcal{O}_{K_j}^{\times,\new}
\subset
\bigcup_{i\in S}\mathcal{U}_i,
\qquad\text{for every } 1\leq j\leq m.
$$
Then every coordinate of every tuple in $U$ lies in $\bigcup_{i\in S}\mc{U}_i$. Consequently, the statistic $Z_U(P)$
depends only on the $S$-distinguished factor of $P$. For every $n\geq 1$, set
$$
\mathcal{E}_n:=\{\bar P_n=0\}=
\{\xi_0,\ldots,\xi_n\in p\Z_p\}
$$
and
$$
\mathcal{E}_n^{\Haar}:=\{\bar P_n^{\Haar}=0\}
=\{\eta_0,\ldots,\eta_n\in p\Z_p\}.
$$
By independence and the balancedness assumptions,
\begin{equation}
\label{eq: exceptional reductions theorem 1.2}
\mathbf{P}(\mathcal{E}_n)\leq(1-\epsilon)^{n+1},\qquad\mathbf{P}(\mathcal{E}_n^{\Haar})=p^{-(n+1)}.
\end{equation}
On $\mathcal{E}_n^c$ and $(\mathcal{E}_n^{\Haar})^c$, let
$P_n^S$ and $P_n^{\Haar,S}$ be the corresponding $S$-distinguished factors. On the exceptional events, we set
$$
P_n^S=P_n^{\Haar,S}:=1.
$$
We first compare their resultant distributions. Fix
$Z\in\mathscr{P}^S$. On the event $\mathcal{E}_n^c$, \Cref{prop: properties of resultant} gives
$$
\val(\Res(P_n,Z))=\val(\Res(P_n^S,Z))+
\val(\Res(P_n/P_n^S,Z)).
$$
The reductions modulo $p$ of $P_n/P_n^S$ and $Z$ have no common irreducible factor. Hence, by
\Cref{prop: relatively prime resultant},
$$
\val(\Res(P_n/P_n^S,Z))=0.
$$
Therefore,
$$
\val(\Res(P_n^S,Z))=\val(\Res(P_n,Z))
$$
on $\mathcal{E}_n^c$. Similarly,
$$
\val(\Res(P_n^{\Haar,S},Z))
=
\val(\Res(P_n^{\Haar},Z))
$$
on $(\mathcal{E}_n^{\Haar})^c$.

Since $1\notin S$, the constant term of every
$Z\in\mathscr{P}^S$ lies in $\Z_p^\times$. It follows from
\Cref{thm: random polynomial resultant convergence} that
$$
\val(\Res(P_n,Z))
\quad\text{and}\quad
\val(\Res(P_n^{\Haar},Z))
$$
converge weakly to the same limiting random variable. Moreover, by
\eqref{eq: exceptional reductions theorem 1.2}, the probabilities of the events on which these variables may differ from
$$
\val(\Res(P_n^S,Z))
\quad\text{and}\quad
\val(\Res(P_n^{\Haar,S},Z)),
$$
respectively, tend to zero. Consequently,
$$
\val(\Res(P_n^S,Z))
\quad\text{and}\quad
\val(\Res(P_n^{\Haar,S},Z))
$$
also converge weakly to the same limiting random variable.

By \Cref{prop: degree moment tightness random polynomial}, we have
$$
\lim_{D\to\infty}
\sup_{n\geq 1}
\E\left[
(\deg P_n^S)^m
\mathbf{1}_{{\deg P_n^S>D}}
\right]=0
$$
and
$$
\lim_{D\to\infty}
\sup_{n\geq 1}
\E\left[
(\deg P_n^{\Haar,S})^m
\mathbf{1}_{{\deg P_n^{\Haar,S}>D}}
\right]=0.
$$
In particular, both degree sequences are tight. We may therefore apply
\Cref{thm: resultant distribution convergence introduction} to obtain random elements
$$
P_\infty^S,\ P_\infty^{\Haar,S}\in\mathscr{P}^S
$$
such that
$$
P_n^S
\overset{d}{\longrightarrow}
P_\infty^S
\qquad\text{and}\qquad
P_n^{\Haar,S}
\overset{d}{\longrightarrow}
P_\infty^{\Haar,S}.
$$
Moreover, for every fixed $Z\in\mathscr{P}^S$,
$$
\val(\Res(P_n^S,Z))
\overset{d}{\longrightarrow}
\val(\Res(P_\infty^S,Z))
$$
and
$$
\val(\Res(P_n^{\Haar,S},Z))
\overset{d}{\longrightarrow}
\val(\Res(P_\infty^{\Haar,S},Z)).
$$
Since the two sequences on the left converge to the same distribution, we have
$$
\val(\Res(P_\infty^S,Z))
\overset{d}{=}
\val(\Res(P_\infty^{\Haar,S},Z))
$$
for every $Z\in\mathscr{P}^S$. Consequently,
\Cref{thm: resultant distribution determination introduction} gives
\begin{equation}
\label{eq: general and Haar limiting factors equal}
P_\infty^S
\overset{d}{=}
P_\infty^{\Haar,S}.
\end{equation}

We next prove that $P_\infty^{\Haar,S}$ is almost surely squarefree. Since
$$
P_\infty^{\Haar,S}=\prod_{i\in S}
P_\infty^{\Haar,\{i\}},
$$
where the factors are supported on distinct lifted subspaces, it suffices to prove that, for every fixed $i\in S$, the $\{i\}$-distinguished factor
$P_\infty^{\Haar,\{i\}}$ is almost surely squarefree.
For every $r\geq 0$, we use the notation from \Cref{lem: Haar distinguished factor fixed degree law} that
$$
\mathcal{C}_{i,r}=
\left\{
Q\in\Z_p[x]:
Q\text{ is monic},
\deg Q=rd_i,
\bar Q=F_i^r
\right\},
$$
and let $\mu_{i,r}$ denote its normalized Haar probability measure. Since
$\mathcal{C}_{i,r}$ is a clopen component of
$\mathscr{P}^{\{i\}}$, weak convergence gives
$$
\mathbf{P}\left(
P_n^{\Haar,\{i\}}\in\mathcal{C}_{i,r}
\right)
\longrightarrow
\mathbf{P}\left(
P_\infty^{\Haar,\{i\}}\in\mathcal{C}_{i,r}
\right).
$$
Set
$$
a_{i,r}:=\mathbf{P}\left(
P_\infty^{\Haar,\{i\}}\in\mathcal{C}_{i,r}
\right)=\mathbf{P}\left(
\deg P_\infty^{\Haar,\{i\}}=rd_i
\right).
$$
Thus $(a_{i,r})_{r\geq 0}$ is the limiting degree distribution of the
$\{i\}$-distinguished factor.

By \Cref{lem: Haar distinguished factor fixed degree law}, conditional on
$$
P_n^{\Haar,\{i\}}\in\mathcal{C}_{i,r}
\quad\text{and}\quad
\bar P_n^{\Haar}\neq 0,
$$
the polynomial $P_n^{\Haar,\{i\}}$ has distribution $\mu_{i,r}$. Since
$$
\mathbf{P}\left(\bar P_n^{\Haar}=0\right)
=p^{-(n+1)}\longrightarrow 0,
$$
passing to the limit shows that the restriction of the law of
$P_\infty^{\Haar,\{i\}}$ to $\mathcal{C}_{i,r}$ is
$a_{i,r}\mu_{i,r}$. Equivalently,
$$
\mathcal{L}\left(P_\infty^{\Haar,\{i\}}\right)
=
\sum_{r\geq 0}a_{i,r}\mu_{i,r}.
$$
In particular, conditional on
$$
\deg P_\infty^{\Haar,\{i\}}=rd_i,
$$
provided $a_{i,r}>0$, the polynomial
$P_\infty^{\Haar,\{i\}}$ is Haar distributed on
$\mathcal{C}_{i,r}$.

For $r\geq 1$, the discriminant is a nonzero polynomial in the
$rd_i$ non-leading coefficients of a monic polynomial of degree $rd_i$.
Its zero set therefore has $\mu_{i,r}$-measure zero. The case $r=0$ is immediate, since
$\mathcal{C}_{i,0}={1}$. Consequently,
$$
\mathbf{P}\left(
P_\infty^{\Haar,\{i\}}
\text{ has a repeated root}
\right)=\sum_{r\geq 0}a_{i,r}
\mu_{i,r}\left(
\left\{
Q\in\mathcal{C}_{i,r}:
Q\text{ has a repeated root}
\right\}\right)=0.
$$
Since $S$ is finite, every factor
$P_\infty^{\Haar,\{i\}}$, $i\in S$, is simultaneously squarefree almost surely. Moreover, factors supported on distinct lifted subspaces cannot share a root. Therefore,
$$
\mathbf{P}\left(
P_\infty^{\Haar,S}
\text{ has a repeated root}
\right)=0.
$$
Together with \eqref{eq: general and Haar limiting factors equal}, this gives
\begin{equation}
\label{eq: limiting factors squarefree}
\mathbf{P}\left(
P_\infty^S\text{ has a repeated root}
\right)=\mathbf{P}\left(
P_\infty^{\Haar,S}\text{ has a repeated root}
\right)=0.
\end{equation}

The root tuples counted by $Z_U$ lie entirely in the lifted subspaces indexed by $S$. Hence, on $\mathcal{E}_n^c$ we have $Z_U(P_n)=Z_U(P_n^S)$, and on $(\mathcal{E}_n^{\Haar})^c$,
$$
Z_U(P_n^{\Haar})=Z_U(P_n^{\Haar,S}).
$$
On the exceptional events, the corresponding root-tuple counts are bounded by $n^m$. Thus,
$$
\left|
\E[Z_U(P_n)]-
\E[Z_U(P_n^S)]
\right|
\leq
n^m\mathbf{P}(\mathcal{E}_n)
\leq
n^m(1-\epsilon)^{n+1}
\longrightarrow 0
$$
and
$$
\left|
\E[Z_U(P_n^{\Haar})]
-
\E[Z_U(P_n^{\Haar,S})]
\right|
\leq
n^mp^{-(n+1)}
\longrightarrow 0.
$$

Applying
\Cref{thm: expected root statistics convergence}, using
\eqref{eq: limiting factors squarefree}, gives
$$
\lim_{n\to\infty}\E[Z_U(P_n^S)]
=\E[Z_U(P_\infty^S)]
$$
and
$$
\lim_{n\to\infty}\E[Z_U(P_n^{\Haar,S})]
=\E[Z_U(P_\infty^{\Haar,S})].
$$
Since the two limiting polynomials have the same distribution, we conclude that
$$
\lim_{n\to\infty}\E[Z_U(P_n)]
=\lim_{n\to\infty}\E[Z_U(P_n^{\Haar})].
$$

Finally, since each
$\mc{O}_{K_j}^{\times,\new}$ is clopen in
$\mc{O}_{K_j}^{\new}$, the set $U$ is clopen, and hence open, in
$$
\mc{O}_{K_1}^{\new}
\times\cdots\times
\mc{O}_{K_m}^{\new}.
$$ 
Hence, we can apply \cite[Theorem 5.8]{caruso2022zeroes} to prove that for every sufficiently large $n$,
$$
\E[Z_U(P_n^{\Haar})]
=\int_U\rho_{K_1,\ldots,K_m}^{(\infty),\poly}(x_1,\ldots,x_m)dx_1\cdots dx_m.
$$
Therefore,
\begin{equation}
\label{eq: root universality clopen}
\lim_{n\to\infty}\E[Z_U(P_n)]
=\int_U\rho_{K_1,\ldots,K_m}^{(\infty),\poly}(x_1,\ldots,x_m)dx_1\cdots dx_m.
\end{equation}
This completes the proof.
\end{proof}

\begin{rmk}
We mention ongoing joint work with Asvin G and Van Peski on explicit formulas for the Haar correlation functions considered above. For distinct points $x_1,\ldots,x_m\in\Z_p$, we evaluate the $m$-point correlation function for arbitrary $m$ by a finite expression determined by the ultrametric tree of the configuration. This yields explicit formulas for all limiting factorial moments of the number of roots in $\Z_p$, and hence its limiting distribution in the Haar coefficient model. Analogous methods apply to configurations in finite extensions of $\Q_p$. Combined with \Cref{thm: root correlation universality}, these formulas also describe the corresponding universal limiting statistics for roots of absolute value $1$ of polynomials with independent $\epsilon$-balanced coefficients.
\end{rmk}

%% file: references.bib
@article{wood2019random,
  title={Random integral matrices and the Cohen-Lenstra heuristics},
  author={Wood, Melanie Matchett},
  journal={American Journal of Mathematics},
  volume={141},
  number={2},
  pages={383--398},
  year={2019},
  publisher={Johns Hopkins University Press}
}

@article{nguyen2025local,
  author  = {Nguyen, Hoi H. and Wood, Melanie Matchett},
  title   = {Local and global universality of random matrix cokernels},
  journal = {Mathematische Annalen},
  volume  = {391},
  number  = {4},
  pages   = {5117--5210},
  year    = {2025}
}

@article{lee2023universality,
  author  = {Lee, Jungin},
  title   = {Universality of the cokernels of random $p$-adic Hermitian matrices},
  journal = {Transactions of the American Mathematical Society},
  volume  = {376},
  number  = {12},
  pages   = {8699--8732},
  year    = {2023}
}

@article{jung2026sharp,
  author = {Jung, Jiwan and Lee, Jungin and Yu, Myungjun},
  title = {Sharp threshold for universality of cokernels of classical random
           matrix models over the $p$-adic integers},
  journal = {arXiv preprint arXiv:2603.12879},
  year = {2026}
}

@article{dyson1962statisticalI,
  author  = {Dyson, Freeman J.},
  title   = {Statistical Theory of the Energy Levels of Complex Systems. I},
  journal = {Journal of Mathematical Physics},
  volume  = {3},
  year    = {1962}
}

@article{dyson1962statisticalIII,
  author  = {Dyson, Freeman J.},
  title   = {Statistical Theory of the Energy Levels of Complex Systems. III},
  journal = {Journal of Mathematical Physics},
  volume  = {3},
  year    = {1962}
}

@book{mehta2004random,
  author    = {Mehta, Madan Lal},
  title     = {Random Matrices},
  edition   = {3},
  publisher = {Elsevier/Academic Press},
  year      = {2004}
}

@article{macchi1975coincidence,
  author  = {Macchi, Odile},
  title   = {The Coincidence Approach to Stochastic Point Processes},
  journal = {Advances in Applied Probability},
  volume  = {7},
  pages   = {83--122},
  year    = {1975}
}

@article{soshnikov2000determinantal,
  author  = {Soshnikov, Alexander},
  title   = {Determinantal Random Point Fields},
  journal = {Russian Mathematical Surveys},
  volume  = {55},
  number  = {5},
  pages   = {923--975},
  year    = {2000}
}

@article{frechet1931proof,
  title={A proof of the generalized second-limit theorem in the theory of probability},
  author={Fr{\'e}chet, Maurice and Shohat, James},
  journal={Transactions of the American Mathematical Society},
  volume={33},
  number={2},
  pages={533--543},
  year={1931}
}

@inproceedings{stieltjes1894recherches,
  title={Recherches sur les fractions continues},
  author={Stieltjes, T-J},
  booktitle={Annales de la Facult{\'e} des sciences de Toulouse: Math{\'e}matiques},
  volume={8},
  number={4},
  pages={J1--J122},
  year={1894}
}

@article{tao2015local,
  title={Local universality of zeroes of random polynomials},
  author={Tao, Terence and Vu, Van},
  journal={International Mathematics Research Notices},
  volume={2015},
  number={13},
  pages={5053--5139},
  year={2015},
  publisher={Oxford University Press}
}

@article{nguyen2022roots,
  title={Roots of random functions: a framework for local universality},
  author={Nguyen, Oanh and Vu, Van},
  journal={American Journal of Mathematics},
  volume={144},
  number={1},
  pages={1--74},
  year={2022},
  publisher={Johns Hopkins University Press}
}

@article{kabluchko2014asymptotic,
  title={Asymptotic distribution of complex zeros of random analytic functions},
  author={Kabluchko, Zakhar and Zaporozhets, Dmitry},
  journal={The Annals of Probability},
  pages={1374--1395},
  year={2014},
  publisher={JSTOR}
}

@article{tao2011random,
  title={Random matrices: universality of local eigenvalue statistics},
  author={Tao, Terence and Vu, Van},
  journal={Acta {M}athematica},
  year={2011}
}

@article{breuillard2019irreducibility,
  title={Irreducibility of random polynomials of large degree},
  author={Breuillard, Emmanuel and Varj{\'u}, P{\'e}ter P},
  journal={Acta Mathematica},
  year={2019}
}

@article{wigner1958distribution,
  title={On the distribution of the roots of certain symmetric matrices},
  author={Wigner, Eugene P},
  journal={Annals of Mathematics},
  volume={67},
  number={2},
  pages={325--327},
  year={1958},
  publisher={JSTOR}
}

@incollection{wigner1993characteristic2,
  title={Characteristic vectors of bordered matrices with infinite dimensions II},
  author={Wigner, Eugene P},
  booktitle={The Collected Works of Eugene Paul Wigner: Part A: The Scientific Papers},
  pages={541--545},
  year={1993},
  publisher={Springer}
}

@incollection{wigner1993characteristic1,
  title={Characteristic vectors of bordered matrices with infinite dimensions i},
  author={Wigner, Eugene P},
  booktitle={The Collected Works of Eugene Paul Wigner: Part A: The Scientific Papers},
  pages={524--540},
  year={1993},
  publisher={Springer}
}

@article{he2023universality,
  title={Universality for low-degree factors of random polynomials over finite fields},
  author={He, Jimmy and Tuan Pham, Huy and Wenqiang Xu, Max},
  journal={International Mathematics Research Notices},
  volume={2023},
  number={17},
  pages={14752--14794},
  year={2023},
  publisher={Oxford University Press}
}

@article{ibragimov1971expected,
  title={On the expected number of real zeros of random polynomials I. Coefficients with zero means},
  author={Ibragimov, Ildar A and Maslova, Nina B},
  journal={Theory of Probability \& Its Applications},
  volume={16},
  number={2},
  pages={228--248},
  year={1971},
  publisher={SIAM}
}

@article{erdos1956number,
  title={On the number of real roots of a random algebraic equation},
  author={Erd{\"o}s, Paul and Offord, A Cyril},
  journal={Proceedings of the London Mathematical Society},
  volume={3},
  number={1},
  pages={139--160},
  year={1956},
  publisher={Wiley Online Library}
}

@article{kac1943average,
  title={On the average number of real roots of a random algebraic equation},
  author={Kac, Mark},
  journal={Bulletin of the American Mathematical Society},
  year={1943}
}

@article{shen2026eigenvalues,
  title={Eigenvalues of $p$-adic random matrices},
  author={Shen, Jiahe and Van Peski, Roger},
  journal={arXiv preprint arXiv:2601.06283},
  year={2026}
}

@book{neukirch2013algebraic,
  title={Algebraic number theory},
  author={Neukirch, J{\"u}rgen},
  volume={322},
  year={2013},
  publisher={Springer Science \& Business Media}
}

@article{van2021limits,
  title={Limits and fluctuations of p-adic random matrix products},
  author={Van Peski, Roger},
  journal={Selecta Mathematica},
  volume={27},
  pages={1--71},
  year={2021},
  publisher={Springer}
}

@article {caruso2022zeroes,
    AUTHOR = {Caruso, Xavier},
     TITLE = {Where are the zeroes of a random {$p$}-adic polynomial?},
   JOURNAL = {Forum Math. Sigma},
  FJOURNAL = {Forum of Mathematics. Sigma},
    VOLUME = {10},
      YEAR = {2022},
     PAGES = {Paper No. e55, 41},
      ISSN = {2050-5094},
   MRCLASS = {11S05 (11K41 12E05 60E05)},
  MRNUMBER = {4454337},
MRREVIEWER = {Kazunari\ Sugiyama},
       DOI = {10.1017/fms.2022.27},
       URL = {https://doi.org/10.1017/fms.2022.27},
}

@article{cheong2023distribution,
  title={The distribution of the cokernel of a polynomial evaluated at a random integral matrix},
  author={Cheong, Gilyoung and Yu, Myungjun},
  journal={arXiv preprint arXiv:2303.09125},
  year={2023},
    note={To appear in \emph{American Journal of Mathematics}}
}

@article{wood2017distribution,
  title={The distribution of sandpile groups of random graphs},
  author={Wood, Melanie Matchett},
  journal={Journal of the American Mathematical Society},
  volume={30},
  number={4},
  pages={915--958},
  year={2017}
}

@article{ellenberg2011modeling,
  title={Modeling $\lambda$-invariants by p-adic random matrices},
  author={Ellenberg, Jordan S and Jain, Sonal and Venkatesh, Akshay},
  journal={Communications on pure and applied mathematics},
  volume={64},
  number={9},
  pages={1243--1262},
  year={2011},
  publisher={Wiley Online Library}
}

@article{sawin2022moment,
  title={The moment problem for random objects in a category},
  author={Sawin, Will and Wood, Melanie Matchett},
  journal={arXiv preprint arXiv:2210.06279},
  year={2022}
}

@article{nguyen2022random,
  title={Random integral matrices: universality of surjectivity and the cokernel},
  author={Nguyen, Hoi H and Wood, Melanie Matchett},
  journal={Inventiones mathematicae},
  volume={228},
  number={1},
  pages={1--76},
  year={2022},
  publisher={Springer}
}

@article{erdHos2012bulk,
  title={Bulk universality for generalized {W}igner matrices},
  author={Erd{\H{o}}s, L{\'a}szl{\'o} and Yau, Horng-Tzer and Yin, Jun},
  journal={Probability Theory and Related Fields},
  volume={154},
  number={1-2},
  pages={341--407},
  year={2012},
  publisher={Springer}
}

@article{shmueli2023expected,
  title={The expected number of roots over the field of {$p$}-adic numbers},
  author={Shmueli, Roy},
  journal={International Mathematics Research Notices},
  volume={2023},
  number={3},
  pages={2543--2571},
  year={2023},
  publisher={Oxford University Press}
}
